\documentclass[11pt,reqno]{article}
\usepackage[margin=1in]{geometry}
\usepackage{amsmath, amsthm, amssymb}
\usepackage{hyperref}
\hypersetup{colorlinks=true, pdfstartview=FitV, linkcolor=blue, citecolor=blue, urlcolor=blue}
\usepackage{times}
\usepackage{subcaption}
\usepackage{mathtools}
\usepackage{bm}
\usepackage{mathabx}
\usepackage{geometry}
\usepackage{enumerate}
\usepackage{authblk}

 \newtheorem{Thm}{Theorem}[section]
 \newtheorem{Prop}[Thm]{Proposition}
 \newtheorem{Rmk}[Thm]{Remark}
 \newtheorem{Lem}[Thm]{Lemma}

\newcommand{\R}{\mathbb{R}}
\newcommand{\N}{\mathbb{N}}
\newcommand{\bR}{\bm{R}}
\newcommand{\bU}{\bm{U}}
\newcommand{\bV}{\bm{V}}
\newcommand{\cE}{{\mathcal E}}

\newcommand{\A}{{\sf A}}
\newcommand{\D}{{\sf D}}
\newcommand{\M}{{\sf M}}
\renewcommand{\S}{{\sf S}}

\def\baru {\underline u}
\def\barh {\underline h}
\def\barrho {\underline\rho}
\def\barH {\underline H}
\def\barU {\underline U}
\def\barM {\underline M}

\DeclareMathOperator{\Id}{Id}
\DeclareMathOperator{\Fr}{Fr}
\DeclareMathOperator{\dd}{{\rm d}\!}
\DeclareMathOperator*{\esssup}{ess\,sup}
\DeclareMathOperator*{\essinf}{ess\,inf}

\numberwithin{equation}{section}

\begin{document}

\title{Relaxing the sharp density stratification and
	columnar motion assumptions in layered shallow water systems}

\author[1]{Mahieddine Adim}
\author[2]{Roberta Bianchini}
\author[1,3]{Vincent Duchêne}
\affil[1]{IRMAR, Univ. Rennes , F-35000 Rennes, France.}
\affil[2]{IAC, Consiglio Nazionale delle Ricerche, 00185 Rome, Italy.}
\affil[3]{CNRS. Corresponding author: \url{vincent.duchene@univ-rennes.fr}}
\date{\today}

\maketitle 

\begin{abstract}
	We rigorously justify the bilayer shallow-water system as an approximation to the hydrostatic Euler equations in situations where the flow is density-stratified with close-to-piecewise constant density profiles, and close-to-columnar velocity profiles. Our theory accommodates with continuous stratification, so that admissible deviations from bilayer profiles are not pointwise small. This leads us to define refined approximate solutions that are able to describe at first order the flow in the pycnocline. Because the hydrostatic Euler equations are not known to enjoy suitable stability estimates, we rely on thickness-diffusivity contributions proposed by Gent and McWilliams. Our strategy also applies to one-layer and multilayer frameworks.  
\end{abstract}

\section{Introduction}

\paragraph{Motivation} The bilayer shallow water system is a standard model for the description of internal waves in density-stratified flows in situations where the density distribution is such that the fluid can be approximately described as two layers with almost-constant densities separated by a thin a pycnocline; see {\em e.g.} ~\cite[Chap.~6]{Gill82}. In addition to this {\em sharp stratification} assumption, the formal derivation of the bilayer shallow water system relies on two additional ingredients. Firstly, the internal pressure is assumed to be {\em hydrostatic}, that is the pressure-gradient force balances the external force due to gravity. Secondly, the flow velocity is assumed to be {\em columnar}, that is the horizontal velocity of fluid particles is constant with respect to the vertical variable within each layer.
Of course the validity of bilayer shallow water system relies on the expectation that, if originally approximately satisfied, these three assumptions remain accurate on a relevant timescale as the flow evolves.

The rigorous justification of the hydrostatic assumption in the shallow water regime ---that is when the typical horizontal wavelength of the flow is large with respect to the vertical depth of the layer--- has been rigorously analyzed either in situations of homogeneous density~\cite{AzeradGuillen2001,LiTiti2019,FurukawaGigaHieberEtAl20,LiTitiYuan21}, with 
smooth density distributions~\cite{PuZhou21,PuZhou22}, or in the bilayer framework~\cite{BonaLannesSaut08,Duchene10,MM4WW}. In the bilayer framework and assuming that the pressure is hydrostatic, the columnar assumption is propagated exactly by the flow. By this we mean that the bilayer shallow water system produces exact solutions to the hydrostatic (incompressible) Euler equations with a density and horizontal velocity distributions which are piecewise constant with respect to the vertical variable. This statement is made explicit below.

In this work we investigate solutions to the hydrostatic Euler equations in the vicinity of such solutions, that is relaxing the sharp stratification as well as columnar motion assumptions. We prove that for initial data suitably close to the bilayer framework, the emerging solutions to the hydrostatic Euler equations remain close to the solution predicted by the bilayer shallow water system on a relevant timescale.

This task is made difficult in part because we lack good stability estimates for the hydrostatic Euler equations in the presence of density stratification. For that matter, as in our previous works~\cite{BianchiniDuchene,Adim}, we rely on the regularizing properties of thickness-diffusivity terms proposed by Gent and McWilliams~\cite{GentMcWilliams90} so as to model the effective contributions of geostrophic eddy correlations in non-eddy-resolving systems. 

\paragraph{Description of our results} Specifically, the hydrostatic Euler equations we consider take the form
\begin{equation} \label{eq.hydro-intro}
\begin{aligned}
\partial_t  h + \partial_x ((1+  h) (  \baru +  u)) & = \kappa \partial_x^2  h,\\
\partial_t  u + \left( \baru+ u - \kappa \frac{\partial_x  h}{1+ h}\right)\partial_x u +\frac1{\barrho} \partial_x  \Psi &=0, \\
\end{aligned}
\end{equation}
where the Montgomery potential $\Psi$ is given by
\begin{equation}\label{eq.def-mont-intro}
\Psi(t,x,r)= \barrho (r)  \int_{-1}^r   h (t,x, r') \dd r' + \int_r^0  \barrho (r')  h (t,x, r') \dd r'  .
\end{equation}
Here, the equations are formulated using isopycnal coordinates (in particular we assume that the fluid is stratified in the sense that the two-dimensional fluid domain is foliated  through lines of equal density, namely {\em isopycnals}). The variable $h$ represents the deviation of the infinitesimal depth of isopycnals from the reference value $1$, and $ u $ is the deviation of the horizontal velocity of the fluid particles from the reference value $\baru$, and $\barrho$ their density. The unknowns $h,u$ depend on the time $t$, the horizontal space $x$, and the variable $r\in(-1,0)$ referring to the isopycnal line at stake, while $\baru$ and $\barrho$ are given and depend only on $r$. 
The derivation of these equations from the more standard formulation in Eulerian coordinates is described for instance in~\cite{BianchiniDuchene}.%
\footnote{Let us point out that in~\cite{BianchiniDuchene} we choose to label isopycnal lines using the value of the density of fluid particles: $\varrho=\barrho(r)$. Here we use a different convention, so as to set the reference infinitesimal depth of isopycnals at value $\barh(r)=1$. Notice that the change of variable $\varrho=\barrho(r)$ is bijective in the stably stratified situation, {\em i.e.} when $r\mapsto\barrho(r)$ is strictly decreasing, but that our choice in this work allows to consider stratifications that are not strictly monotonic such as homogeneous and layered configurations. Incidentally, let us mention that in~\eqref{eq.hydro-intro} we have set the gravity acceleration to $g=1$ and the total depth of the fluid domain at rest to $\int_{-1}^0\barh(r) \dd r=1$ through suitable rescaling. } 
Finally $\kappa>0$ is the thickness diffusivity coefficient and $u_\star=- \kappa \frac{\partial_x  h}{1+ h}$ is often referred to as the ``bolus velocity''.

The corresponding bilayer shallow water system reads
\begin{equation} \label{eq.SV2-intro}
\begin{cases}
	\partial_t H_s + \partial_x \big((\barH_s +H_s)( \barU_s+U_s)\big) = \kappa\partial_x^2 H_s, \\	
	\partial_t H_b + \partial_x \big((\barH_b +H_b)( \barU_b+U_b)\big) = \kappa\partial_x^2 H_s, \\
	\partial_t U_s + (\barU_s+U_s-\kappa\frac{\partial_x H_s}{\barH_s+H_s})\partial_x U_s +  \partial_x H_s +  \partial_x H_b=0,\\
	\partial_t U_b + (\barU_b+U_b-\kappa\frac{\partial_x H_b}{\barH_b+H_b})\partial_x U_b + \frac{\rho_s}{\rho_b}\partial_x H_s +  \partial_x H_b=0.
\end{cases}
\end{equation}
Here, $H_s$ (resp. $H_b$) represents the deviation of the depth of the upper (resp. lower) layer from the reference constant value $\barH_s$ (resp. $\barH_b$), and $ U_s $ (resp. $U_b$) is the deviation of the horizontal velocity within the upper (resp. lower) layer from the reference constant value $\barU_s$ (resp. $\barU_b$). Notice $H_s,H_b,U_s,U_b$ depend only on the time and horizontal space variables, bringing about a relative ease of use of the bilayer model. We denote $\rho_s$ (resp. $\rho_b$) the constant density of fluid particles in the upper (resp. lower) layers. 
Finally, $\kappa>0$ is again the thickness diffusivity coefficient.

As mentioned above, solutions to~\eqref{eq.SV2-intro} provide exact solutions to~\eqref{eq.hydro-intro}.
Specifically, if we denote
\begin{equation}\label{eq.bilayer-to-continuous-intro}
\begin{aligned}
\barrho_{\rm bl}(r) &= \rho_s {\bf 1}_{(-\barH_s,0)}(r) +\rho_b{\bf 1}_{(-1,-\barH_s)}(r),\\
\baru_{\rm bl}(r) 	&= \barU_s {\bf 1}_{(-\barH_s,0)}(r) +\barU_b{\bf 1}_{(-1,-\barH_s)}(r),\\
u_{\rm bl}(\cdot,r)  		&=  U_s {\bf 1}_{(-\barH_s,0)}(r) + U_b{\bf 1}_{(-1,-\barH_s)}(r), \\
h_{\rm bl}(\cdot,r) 		&= \frac{H_s}{\barH_s} {\bf 1}_{(-\barH_s,0)}(r)  +\frac{H_b}{\barH_b} {\bf 1}_{(-1,\barH_s)}(r),
\end{aligned}
\end{equation}
where $(\rho_s,\rho_b,\barH_s,\barH_b,\barU_s,\barU_b,H_s,H_b,U_s,U_b)$ is a solution to~\eqref{eq.SV2-intro}, then $(\barrho_{\rm bl},\baru_{\rm bl},h_{\rm bl},u_{\rm bl})$ is a solution to~\eqref{eq.hydro-intro}-\eqref{eq.def-mont-intro}. In this work we shall compare these solutions with the ones emerging from profiles satisfying~\eqref{eq.bilayer-to-continuous-intro} only approximately.
Specifically, our results are twofold. 
\begin{enumerate}[i.]
	\item \label{R.1} We prove that strong solutions to the bilayer shallow water system~\eqref{eq.SV2-intro} emerge from sufficiently regular initial data satisfying some hyperbolicity conditions.
	\item \label{R.2}We prove that strong solutions to the hydrostatic Euler equations~\eqref{eq.hydro-intro} emerge from profiles close to the piecewise constant profiles given by~\eqref{eq.bilayer-to-continuous-intro}, and that these solutions remain close to the bilayer solutions.
\end{enumerate}

Notice that to accommodate our aim of comparing solutions with a piecewise constant density distribution with a solution with a continuous density distribution, we need to consider deviations that can be pointwise large in small regions, that is the pycnocline. This demands to weaken the topology measuring the size of deviations in (\ref{R.2}). Yet a control in the strong $L^\infty$ topology associated with a Banach algebra turns out to be necessary to secure suitable convergence estimates. Our strategy then relies on the construction of a refined approximate solution, which is proved to be close to the corresponding solution to the hydrostatic Euler equation in a strong topology, and close to the bilayer solution in a weaker topology. Hence this refined approximate solution improves the description of the exact solution within the pycnocline.

A second important remark is that while the contribution of thickness diffusivity is essential to our stability estimates, we wish to control and compare solutions on a time interval which is uniform with respect to the thickness diffusivity parameter, $0<\kappa\leq 1$. The dependency on the thickness parameter will appear only as a restriction on the size of admissible deviations. Concerning (\ref{R.1}) this is made possible by the well-known fact that the bilayer shallow-water system is well-posed (under some hyperbolicity conditions) when $\kappa=0$ (see~\cite{Ovsjannikov79}). Yet obtaining the corresponding result for $\kappa>0$ is not straightforward and demands to use finely the structure of the thickness diffusivity parameters, following the ``two-velocity'' strategy developed in the context of the BD-entropy (see {\em e.g.}~\cite{BreschDesjardinsZatorska15}). Concerning (\ref{R.2}) we use the existence of the bilayer solution and consequently the existence of the refined approximate solution to bootstrap the control of sufficiently close solutions to the hydrostatic Euler equations on the relevant timescale. For that purpose we strongly use the regularizing effect of thickness diffusivity contributions, but any non-uniformity with respect to the diffusivity parameter, $\kappa$, can be balanced through the smallness of the deviations.

Let us point out that the result (\ref{R.2}) applies to {\em any} given (sufficiently regular with respect to the time and horizontal space variables) solution to the hydrostatic Euler equations. Hence our work provides the same stability estimates around other solutions, constructed for instance in the framework of multiple layers and/or simple waves. 

\paragraph{Related literature} Several existing works discuss the matter of modeling thin pycnoclines through the bilayer framework. Let us first recall that in the bilayer framework, non-hydrostatic pressure contributions trigger strong Kelvin--Helmholtz instabilities that in particular prevent the well-posedness of the initial-value problem in the absence of any additional regularizing ingredients; see~\cite{KamotskiLebeau05}. Such regularizing ingredients include interfacial tension as proved in~\cite{Lannes13} but this is not expected to be te physically relevant mechanism. In~\cite{BoguckiGarrett93}, Bogucki and Garrett describe and model a scenario of interface-thickening due to mixing triggered by shear instabilities, up to a situation where the Richardson number in the interface becomes compatible with the celebrated Miles~\cite{Miles61} and Howard~\cite{Howard61} stability ion. Recall however that the bilayer shallow water system does not suffer from shear-induced instabilities when shear velocities are sufficiently small. Consistently, the authors in~\cite{DesjardinsLannesSaut21} discuss the simultaneous limits of sharp stratification together with shallow water (which can be considered as a low-frequency or hydrostatic pressure limit). In our work we impose the hydrostatic pressure assumption thus taming shear instabilities, although as pointed out previously the stability of continuously stratified hydrostatic Euler equations is poorly understood in the presence of density variations (see~\cite{Teshukov92,Brenier99,Grenier99,MasmoudiWong12} in the homogeneous framework). Notice also that shear-induced instabilities disappear when restricting the framework to purely traveling waves. In this framework, James~\cite{James01} (improving upon~\cite{Turner81,AmickTurner86}) was able to rigorously justify the sharp stratification limit: writing the bilayer and continuously stratified problems in a unified formulation, James proved the existence of internal traveling waves associated with density stratifications in a small neighborhood (according to the $L^2$ topology) of the bilayer framework which converge towards the bilayer solution in the limit of sharp stratification. Our results are in the same spirit: system \eqref{eq.hydro-intro} is our unified formulation, and the topology controlling the limit of sharp stratification in our work is the $L^1_r$ topology. However, our results admit non-trivial dynamics thanks to the hydrostatic assumption and the presence of thickness diffusivity. 

The propagation of internal waves with thin pycnoclines in relation with bilayer models was also investigated through experiments. In particular Grue {\em et al.}~\cite{GrueJensenRusaasEtAl99} set up precise experiments generating large-amplitude solitary waves and reported the dynamical development of rolls on the trailing side of the largest considered waves in accordance with the mechanism promoted by Bogucki and Garrett, while bilayer models provide very accurate predictions otherwise. Almgren, Camassa and Tiron~\cite{AlmgrenCamassaTiron12} investigate thoroughly this matter through careful numerical simulations and analytical results of asymptotic bilayer models, analyzing the triggering of shear-induced instabilities in the region of maximal displacement as well as their advection into stable regions of the flow.  White and Helfrich~\cite{WhiteHelfrich14} consider internal bores generated by a dam-break, and compare continuously stratified and bilayer models with numerical experiments. From their findings they suggest an improvement on existing bilayer theories. In~\cite{CamassaTiron11}, Camassa and Tiron optimize bilayer models (specifically calibrating the top and bottom densities and the position of the sharp interface) and compare the analytical predictions of the optimized bilayer models with respect to the numerically computed continuously stratified solutions ---considering infinitesimally small waves, internal bores and solitary waves--- showing excellent agreement even in situations of relatively thick pycnoclines. Furthermore they propose a new asymptotic model taking into account thin pycnoclines in view of reconstructing analytically local properties of traveling waves within the pycnocline, which is similar in spirit with the ``refined approximate solution'' that we introduce in this work. Notice the authors consider that ``fully time-dependent models governing the evolution of the pycnocline thickness probably constitute one of the most relevant extensions of the model [they] have introduced'', and we believe that our work provides a partial answer in that respect. 

\vspace{-.1cm}
\paragraph{Outline} This manuscript is structured as follows. In Section~\ref{S.SV2} we study the bilayer shallow water systems. We first recall some known results without thickness diffusivity, and then consider the system with thickness diffusivity contributions. The main result is Proposition~\ref{P.uniform-WP} which provides for any sufficiently regular initial data satisfying some hyperbolicity criterion the existence and control of strong solutions to the bilayer system on a time interval which is uniform with respect to $0<\kappa\leq 1$, while Proposition~\ref{P.convergence-SV2} states the strong convergence as $\kappa\searrow 0$. 
In Section~\ref{S.hydro} we study the hydrostatic Euler equations. We provide first some stability estimates with respect to perturbations of the equations and of the data using suitable distances. As a second step we introduce refined approximate solutions associated with some given reference exact solution and close-by profiles. Building upon these approximate solutions, we prove Proposition~\ref{P.convergence-hydro} which controls the difference between the reference solution and exact solutions to the hydrostatic Euler equations emerging from close-by profiles. Together, Proposition~\ref{P.uniform-WP} and  Proposition~\ref{P.convergence-hydro} provide the announced result that for profiles suitably close to the bilayer framework and satisfying some hyperbolicity criterion 
the emerging solutions to the hydrostatic Euler equations remain close to the solution predicted by the bilayer shallow water system on a relevant timescale. The rigorous statement is displayed in Section~\ref{S.conclusion}, completed with a discussion on analogous statements in the one-layer and multilayer frameworks.

\paragraph{Notations} Let us introduce some notations for functional spaces used in this work.
\begin{itemize}
	\item The spaces $L^p(\R)$ are the standard Lebesgue spaces endowed with the usual norms denoted $\|\cdot\|_{L^p}$.
	\item The spaces $W^{k,p}(\R)$ for $k\in\N$ are the $L^p$-based Sobolev spaces endowed with the usual norms denoted $\|\cdot\|_{W^{k,p}}$.
	\item The spaces $H^s(\R)$ for $s\in\R$ are the $L^2$-based Sobolev spaces endowed with the usual norms denoted $\|\cdot\|_{H^s}$.
	\item Given $I$ a real interval and $X$ a Banach space, $L^p(I;X)$ (respectively $C^n(I;X)$) the space of $p$-integrable (respectively $n$-continuously differentiable) $X$-valued functions, endowed with their usual norms. 
	\item When useful, we provide insights on the variables at stake in aforementioned functional spaces by means of subscripts. For instance for $f:(x,r)\in\R\times(-1,0) \mapsto f(x,r)\in\R$ we may denote
	\[ \|f\|_{L^\infty_r H^s_x}  =\esssup\big(\big\{ \|f(\cdot,r)\|_{H^s} \ , \ r\in (-1,0)\}).\]
	\item We sometimes also use subscripts to provide information on the interval at stake in functional spaces. For instance for $T>0$ and $f:(t,x)\in [0,T]\times\R\mapsto f(t,x)\in\R$ we may denote
	\[ \|f\|_{L^\infty_T H^s_x}  =\esssup\big(\big\{ \|f(t,\cdot)\|_{H^s} \ , \ t\in [0,T]\}).\]
\end{itemize}

\section{The bilayer shallow water system}\label{S.SV2}

In this section we analyze the bilayer shallow water system~\eqref{eq.SV2-intro}. We first consider the case without diffusivity ($\kappa=0$) and recall the hyperbolicity analysis due to Ovsjannikov~\cite{Ovsjannikov79}. We complete it by exhibiting explicit symmetrizers of the system of conservation laws. The standard theory for quasilinear systems then provides the local well-posedness of the initial-value problem, that we state in Proposition~\ref{P.WP-SV}.

Extending such result for the system with diffusivity ($\kappa>0$) {\em uniformly with respect to $\kappa\in(0,1]$} is not as obvious as one could naively think, because the aforementioned symmetrizer behaves poorly with respect to diffusivity contributions. In order to deal with this issue, we exhibit regularization effects stemming from the diffusivity contributions that apply to the {\em total velocity} (that is adding the bolus velocities to the velocity unknowns). This is in the spirit of the BD entropy that arose in the context of the barotropic Euler equations with degenerate viscosities (see~\cite{BreschDesjardinsZatorska15}). We infer a stability result on the linearized system, Lemma~\ref{L.estimate-linearized}, which eventually yields the ``large-time'' ---that is uniform with respect to $\kappa\in(0,1]$--- control of solutions stated in Proposition~\ref{P.uniform-WP}, and their strong convergence towards corresponding solutions to the non-diffusive system as $\kappa \searrow0$ stated in Proposition~\ref{P.convergence-SV2}.

\subsection{The system without thickness diffusivity}\label{S.SV2-nondiffusive}
We consider the system
\begin{equation}
\begin{cases}\label{eq.SV2}
	\partial_t H_s + \partial_x ((\barH_s+H_s)(\barU_s+U_s)) = 0, \\
	\partial_t H_b + \partial_x ((\barH_b+H_b)(\barU_b+U_b)) = 0, \\
	\partial_t U_s + (\barU_s+U_s)\partial_x U_s +  \partial_x H_s +  \partial_x H_b=0,\\
	\partial_t U_b + (\barU_b+U_b)\partial_x U_b + \frac{\rho_s}{\rho_b}\partial_x H_s +  \partial_x H_b=0.
\end{cases}
\end{equation}
We shall also always assume $\rho_s\geq0$ and $\rho_b>0$. Through rescaling and Galilean invariance we can assume without loss of generality that ${\barH_s+\barH_b=1}$ and ${\barU_s+\barU_b=0}$. 

In compact form, the system reads
\[ \partial_t \bU + \A(\underline\bU+\bU)\partial_x \bU=0\]
with $\bU:=(H_s,H_b,U_s,U_b)$, $\underline\bU:=(\barH_s,\barH_b,\barU_s,\barU_b)$ and where we introduce the matrix-valued function
\begin{equation}
\A:(H_s,H_b,U_s,U_b) \in\R^4\mapsto\begin{pmatrix}
U_s & 0 & H_s & 0 \\
0 & U_b & 0 & H_b\\
1 & 1 & U_s & 0 \\
 \frac{\rho_s}{\rho_b} & 1  & 0 &  U_b
\end{pmatrix}.
\end{equation}

The following Lemma concerning the hyperbolicity domain of the bilayer shallow water system is proved in~\cite{Ovsjannikov79,BarrosChoi08,VirissimoMilewski20}.
\begin{Lem} \label{L.hyperbolic}
Let $0<\rho_s<\rho_b$ and $\bU:=(H_s,H_b,U_s,U_b) \in\R^4$ be such that that $H_s,H_b>0$. There exist two values $0 < \Fr_- < \Fr_+$ such that the following holds:
\begin{enumerate}
    \item \label{item.hyperbolic} If $|U_b - U_s| < \sqrt{H_b} \Fr_-$, then there exist four distinct real eigenvalues of the matrix $\A(\bU)$.
    \item If $\sqrt{H_b} \Fr_- < |U_b - U_s| < \sqrt{H_b} \Fr_+$, then there exist two distinct real eigenvalues of the matrix $\A(\bU)$ and two distinct complex conjugate eigenvalues.
    \item If $|U_b - U_s| > \sqrt{H_b} \Fr_+$, then there exist four distinct real eigenvalues of the matrix $\A(\bU)$.
\end{enumerate}
Moreover, $\Fr_-$ and $\Fr_+$ depend only and smoothly on $H_s/H_b \in (0,+\infty)$ and $\rho_s/\rho_b \in (0,1)$.
\end{Lem}
\begin{Rmk} Ovsjannikov~\cite{Ovsjannikov79} ---revisited by Barros and Choi~\cite{BarrosChoi08} and then by Vir\'issimo and Milewski~\cite{VirissimoMilewski20}--- provided a nice geometrical approach to the critical values $\Fr_+$ and $\Fr_-$.	The characteristic polynomial associated to $\A(\bU)$ is
	\[ P(\lambda)=\big((U_b-\lambda)^2-H_b\big)\big((U_s-\lambda)^2-H_s\big)- \frac{\rho_s}{\rho_b} H_sH_b.\]
	Notice that $\lambda\in\R$ is a real root of $P$ if and only if $(p_s,p_b):=(\frac{U_s-\lambda}{\sqrt{H_s}},\frac{U_b-\lambda}{\sqrt{H_b}})$ satisfies the following identities:
	\begin{equation}\label{eq.Hyperbolic-geometric}
	\big(p_s^2-1\big)\big(p_b^2-1\big)= \frac{\rho_s}{\rho_b} ,\qquad p_s\sqrt{H_s}-U_s=p_b\sqrt{H_b}-U_b.
	\end{equation}
	The first equality describes a fourth-order curve parametrized by $\rho_s/\rho_b$ having four axes
	of symmetry and consisting of an inner closed curve 
	and four hyperbolic branches 
	and the second equality describes the straight line with slope  $\sqrt{H_s/H_b}$ and intercept $(U_b-U_s)/\sqrt{H_b}$. In this geometrical approach,  $\Fr_-$ and $\Fr_+$ (and their opposite) are the intercepts of the tangents to the fourth-order curves with slope $\sqrt{H_s/H_b}$.
	
	Figure~\ref{F.hyperbolicity} reproduces the aforementioned curves and straight lines for several parameter values.
\end{Rmk}
	\begin{figure}[htb]
		\begin{subfigure}{.33\textwidth}
			\includegraphics[width=\textwidth]{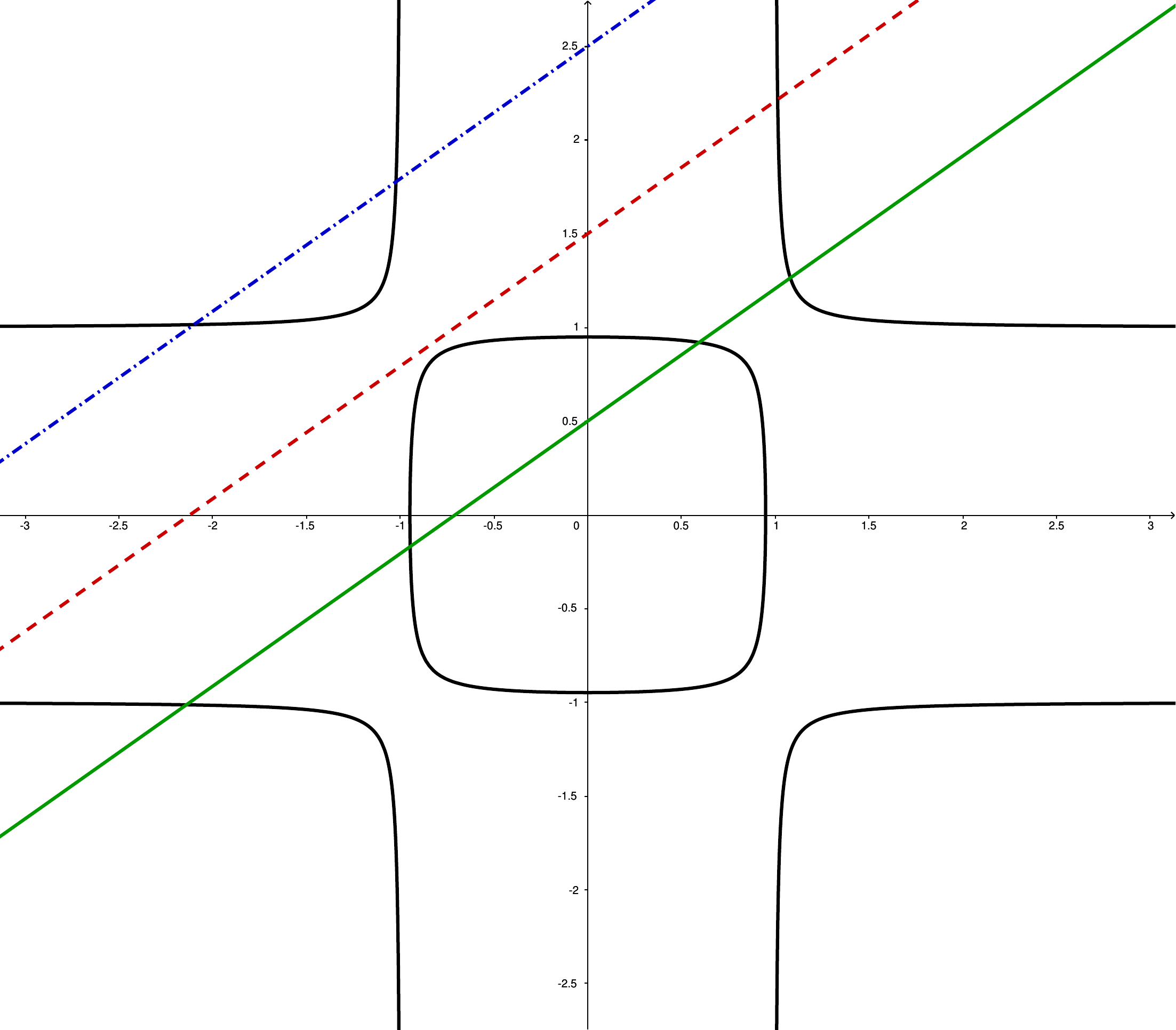}
			\caption{$\rho_s/\rho_b=0.1$}
		\end{subfigure}
		\begin{subfigure}{.33\textwidth}
			\includegraphics[width=\textwidth]{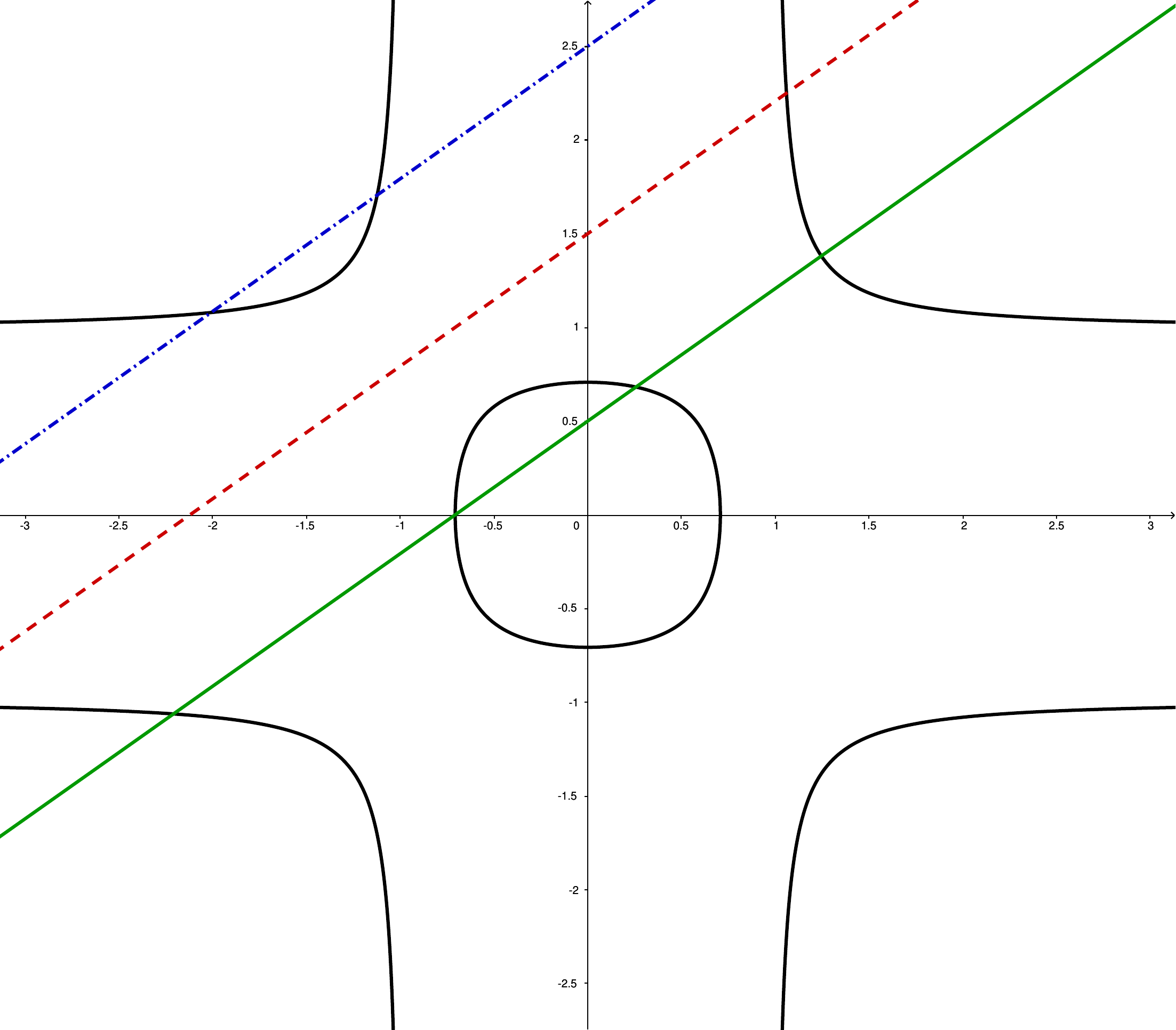}
			\caption{$\rho_s/\rho_b=0.5$}
		\end{subfigure}
		\begin{subfigure}{.33\textwidth}
			\includegraphics[width=\textwidth]{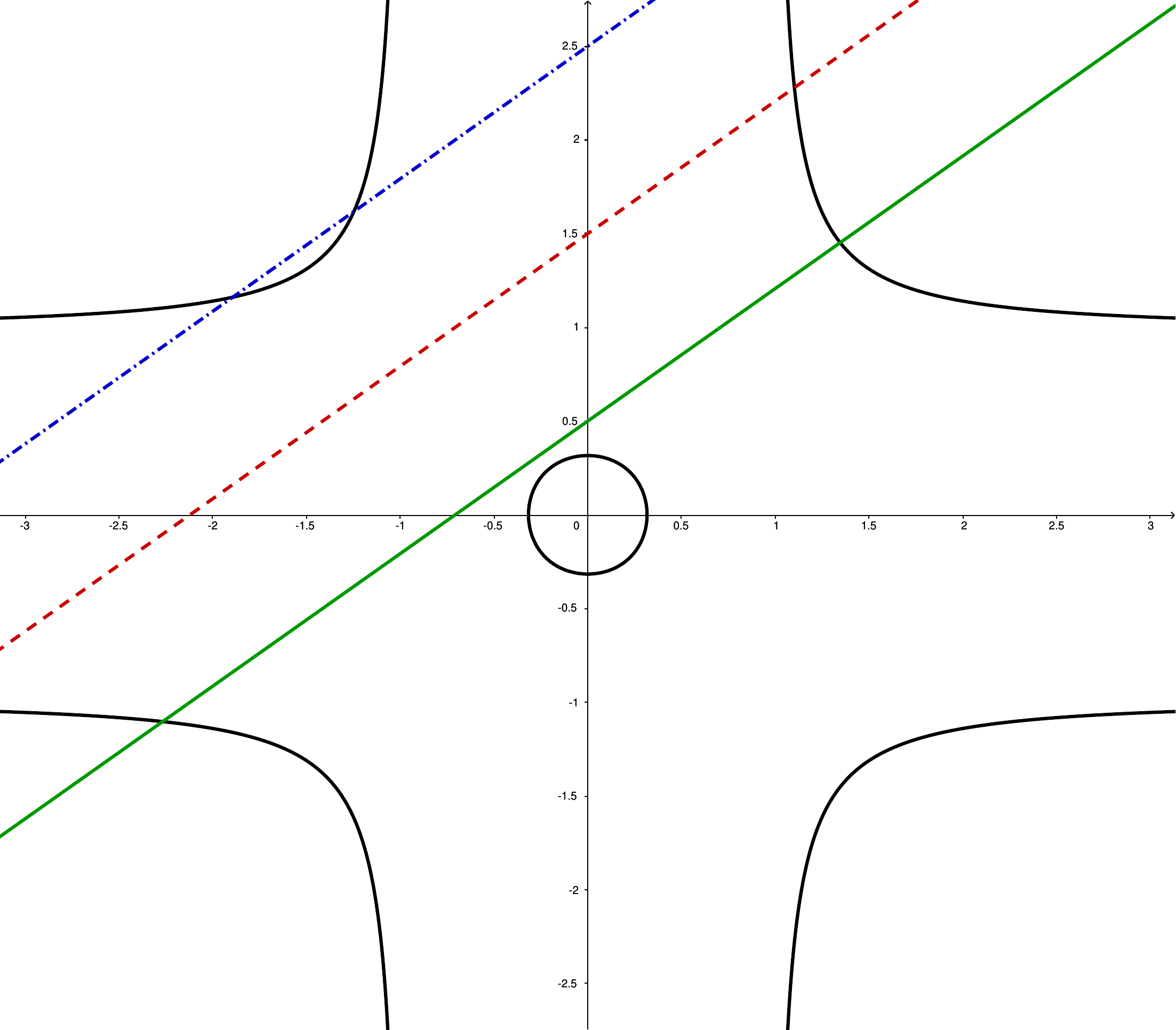}
			\caption{$\rho_s/\rho_b=0.9$}
		\end{subfigure}
		\caption[Geometric interpretation of the hyperbolic domain of the bilayer hydrostatic system]{Solutions to~\eqref{eq.Hyperbolic-geometric} with $H_s=1/3$, $H_b=2/3$, and differnet values for $\rho_s/\rho_b$. Solutions to the quartic equation are in black (plain). Solutions to the linear equation with $(U_b-U_s)/\sqrt{H_b}=1/2$ (green, plain), $(U_b-U_s)/\sqrt{H_b}=3/2$ (red, dashed) and $(U_b-U_s)/\sqrt{H_b}=5/2$ (blue, dot-dashed).}
		\label{F.hyperbolicity}
	\end{figure}

	In this work we restrict our analysis to the hyperbolic domain described by Lemma~\ref{L.hyperbolic}(\ref{item.hyperbolic}). While standard theory for strictly hyperbolic systems guarantees the existence of a symmetrizer to~\eqref{eq.SV2} by using spectral projections~\cite{Benzoni-GavageSerre07}, the following Lemma provides an  (almost) explicit expression for such a symmetrizer.
	
\begin{Lem}\label{L.symmetrizer} Let $0<\rho_s<\rho_b$ and $\bU:=(H_s,H_b,U_s,U_b) \in\R^4$ be such that that $H_s,H_b>0$ and
	\begin{equation}\label{eq.hyperbolic}
		|U_s-U_b|< \sqrt{H_b} \Fr_-
	\end{equation}
	where $\Fr_-=\Fr_-(H_s/H_b,\rho_s/\rho_b)>0$ has been defined in Lemma~\ref{L.hyperbolic}.
	
	There exists $\lambda\in [\min(\{U_s,U_b\}),\max(\{U_s,U_b\})]$ such that, denoting  $U_\ell^\lambda := U_\ell -\lambda\in [-|U_b-U_s|,|U_b-U_s|]$ for $\ell\in\{s,b\}$) the matrix
	\[ \S^\lambda(\bU) := \begin{pmatrix}
		\frac{\rho_s}{\rho_b} & \frac{\rho_s}{\rho_b} & \frac{\rho_s}{\rho_b} U_s^\lambda & 0 \\
		\frac{\rho_s}{\rho_b} & 1 & 0 & U_b^\lambda\\
		\frac{\rho_s}{\rho_b} U_s^\lambda & 0 & \frac{\rho_s}{\rho_b} H_s & 0 \\
		0 & U_b^\lambda  & 0 &  H_b
	\end{pmatrix}\]
	satisfies (i) $\S^\lambda \A$ is symmetric; and (ii) $\S^\lambda$ is symmetric, definite positive.

	Moreover, $\lambda$ can be chosen so that $\frac{\lambda-U_\ell}{\sqrt{H_\ell}} $ (for $\ell\in\{s,b\}$) depends only and smoothly on $H_s/H_b>0$, $ \rho_s/\rho_b\in(0,1)$, and $(U_b-U_s)/ \sqrt{H_b}\in(-\Fr_-,\Fr_-) $.
\end{Lem}
\begin{proof}
	It is straightforward to check that $\S^\lambda$ and $\S^\lambda \A$ are symmetric (and real-valued) for any value of $\lambda\in\R$. In order to prove that $\S^\lambda$ is definite positive for a suitable choice of $\lambda$, we rely on Sylvester's criterion. 
	We obtain the requirements
	\[ H_b>0 ,\quad \frac{\rho_s}{\rho_b} H_sH_b>0, \quad \frac{\rho_s}{\rho_b} H_sH_b-\frac{\rho_s}{\rho_b} H_s(U_b^\lambda)^2>0\]
	 and
	 \[\big(\frac{\rho_s}{\rho_b}\big)^2\Big( \big((U_b^\lambda)^2-H_b\big)\big((U_s^\lambda)^2-H_s\big)-\frac{\rho_s}{\rho_b} H_sH_b\Big)>0.\]
	 The last inequality is equivalent to $P(\lambda)>0$ where $P$ is the aforementioned characteristic polynomial. By Lemma~\ref{L.hyperbolic}, under the condition~\eqref{eq.hyperbolic} there are four distinct real roots to $P$, which we can denote $\lambda_1<\lambda_2<\lambda_3<\lambda_4$ and $P(\lambda)>0$ for any $\lambda\in(\lambda_2,\lambda_3)$. Moreover for all $\lambda\in (\lambda_2,\lambda_3)$, $(p_s^\lambda,p_b^\lambda):=(\frac{U_s-\lambda}{\sqrt{H_s}},\frac{U_b-\lambda}{\sqrt{H_b}})$ belongs to the domain delimited by the inner closed curve, and in particular we have $(p_\ell^\lambda)^2<1$. Hence we find that  for all $\lambda\in(\lambda_2,\lambda_3)$ all principal minors are positive, and hence $\S^\lambda$ is definite positive. 
	 
	 By the standard perturbation theory~\cite{Kato95}, $p_\ell^\lambda$ for $\ell\in\{s,b\}$ and $\lambda\in \{\lambda_2,\lambda_3\}$ depend smoothly on  $H_s/H_b>0$, $ \rho_s/\rho_b\in(0,1)$ and $(U_b-U_s)/ \sqrt{H_b}\in(-\Fr_-,\Fr_-)$. What is more, we can always choose (smoothly) $\lambda \in(\lambda_2,\lambda_3)$ so that $p_s^\lambda >0$ and $p_b^\lambda<0$ when $U_b<U_s$, or $p_s^\lambda <0$ and $p_b^\lambda>0$ when $U_b>U_s$, which corresponds to enforcing $\lambda\in [\min(\{U_s,U_b\}),\max(\{U_s,U_b\})]$. This concludes the proof.
\end{proof}

The following proposition follows from the standard theory on strictly hyperbolic systems (see {\em e.g.}~\cite{Benzoni-GavageSerre07}).
For convenience, we define for $\varsigma\in(0,1)$ a compact subset of the domain of strict hyperbolicity as
\begin{multline}\label{eq.hyperbolicity-condition} \mathfrak p^\varsigma:=\big\{(\rho_s,\rho_b, H_s, H_b,U_s,U_b)\in\R^6 \ :  \\ 
	\varsigma/2\leq\rho_s/\rho_b\leq1-\varsigma/2, \quad \varsigma\leq H_s/H_b \leq \varsigma^{-1}, \quad H_s+H_b\geq \varsigma, \quad \Fr_--\tfrac{|U_b-U_s|}{\sqrt{H_b}}\geq \varsigma \big\} 
\end{multline}
where $\Fr_-=\Fr_-(\rho_s/\rho_b,H_s/H_b)$ is defined in Lemma~\ref{L.hyperbolic}.

\begin{Prop}[Well-posedness]\label{P.WP-SV}
	Let $s\geq s_0> 3/2$, $\varsigma>0$ and $M_0>0$. There exist $C>0$ and $T>0$ such that the following holds. 
	
	For all $(\rho_s,\rho_b, \barH_s,\barH_b,\barU_s,\barU_b)\in\R^6$ such that $\barH_s+\barH_b=1$ and $\barU_s+\barU_b=0$ and $( H_s^0, H_b^0, U_s^0, U_b^0)\in H^s(\R)^4$ such that 
	\[\forall x\in\R,\quad  (\rho_s,\rho_b, \barH_s+ H_s^0(x) , \barH_b+ H_b^0(x),\barU_s+ U_s^0(x),\barU_b+ U_b^0(x))\in\mathfrak p^\varsigma\]
	and
	\[ \| ( H_s^0, H_b^0, U_s^0, U_b^0)\|_{H^{s_0}}\leq M_0\]
	there exists a unique $( H_s, H_b, U_s, U_b)\in C([0,T^\star);H^s(\R)^4)\cap C^1([0,T^\star);H^{s-1}(\R)^4)$  maximal-in-time (classical) solution to~\eqref{eq.SV2} emerging from the initial data $(H_s,H_b,U_s,U_b)\big\vert_{t=0}=( H_s^0, H_b^0, U_s^0, U_b^0)$.
	
	Moreover, one has $T^\star> T/M_0$ and for any $t\in[0,T/M_0]$ one has
		 	\[ \forall x\in\R,\quad  ( \rho_s,\rho_b,\barH_s+ H_s(t,x), \barH_b + H_b(t,x), \barU_s+ U_s(t,x),\barU_b+ U_b(t,x)) \in \mathfrak p^{\varsigma/2}\]
		 	and
	\[  \| ( H_s(t,\cdot), H_b(t,\cdot), U_s(t,\cdot), U_b(t,\cdot))\|_{H^s}\leq C  \| ( H_s^0, H_b^0, U_s^0, U_b^0)\|_{H^{s}}.\]

Moreover, the maximal existence time (resp. the emerging solution in $C([0,T^\star);H^s(\R)^4)$) is a lower semi-continuous (resp. continuous) function of the initial data in $H^s(\R)^4$ and if $T^\star<\infty$ then
\[ \| ( H_s(t,\cdot), H_b(t,\cdot), U_s(t,\cdot), U_b(t,\cdot))\|_{H^{s_0}}\to \infty \text{ as } t\to T^\star.\]
\end{Prop}

\subsection{The system with diffusivity}\label{S.SV2-diffusive}

We now consider the system
\begin{equation}
\begin{cases}\label{eq.SV2-kappa}
	\partial_t H_s + \partial_x ((\barH_s+H_s)(\barU_s+U_s)) = \kappa\partial_x^2 H_s, \\
	\partial_t H_b + \partial_x ((\barH_b+H_b)(\barU_b+ U_b)) = \kappa\partial_x^2 H_b, \\
	\partial_t U_s + (\barU_s+U_s-\kappa\frac{\partial_x H_s}{\barH_s+H_s})\partial_x U_s +  \partial_x H_s +  \partial_x H_b=0,\\
	\partial_t U_b + (\barU_b+U_b-\kappa\frac{\partial_x H_b}{\barH_b+H_b})\partial_x U_b + \frac{\rho_s}{\rho_b}\partial_x H_s +  \partial_x H_b=0.
\end{cases}
\end{equation}

\begin{Prop}[Small time well-posedness]\label{P.WP-small-time-SV2}
Let $s\geq s_0>3/2$, $\varsigma\in(0,1)$, $M_0>0$ and $c>1$. There exists $T>0$ such that the following holds.

For all  $\kappa>0$, for all $(\rho_s,\rho_b, \barH_s,\barH_b,\barU_s,\barU_b)\in\R^6$ such that $\barH_s+\barH_b=1$ and $\barU_s+\barU_b=0$ and for all $( H_s^0, H_b^0, U_s^0, U_b^0)\in H^s(\R)^4$ such that 
\[ 0\leq \rho_s/\rho_b\leq \varsigma^{-1} \quad \text{ and } \quad \forall x\in\R, \quad    \barH_s+ H_s^0(x) \geq \varsigma ,\quad   \barH_b+ H_b^0(x)\geq \varsigma\]
and
\[ \| ( H_s^0, H_b^0, U_s^0, U_b^0)\|_{H^{s_0}}\leq M_0\]
 there exists a unique $( H_s, H_b, U_s, U_b)\in C([0,T^\star);H^s(\R)^4)$  maximal-in-time strong solution to~\eqref{eq.SV2-kappa} emerging from the initial data $(H_s,H_b,U_s,U_b)\big\vert_{t=0}=( H_s^0, H_b^0,U_s^0, U_b^0)$.
 
 Moreover, $T^\star> \kappa T$ and for any $t\in[0,\kappa T]$ one has
 	 	\[ \forall x\in\R,\quad    \barH_s+ H_s(t,x) \geq \varsigma/c ,\quad   \barH_b+ H_b(t,x)\geq \varsigma/c\]
 	 	and
\[ \max(\{ \| ( H_s, H_b, U_s, U_b)\|_{L^\infty(0,t;H^s)},\kappa^{1/2} \| (\partial_x  H_s,\partial_x  H_b)\|_{L^2(0,t;H^s)}\})\leq c  \| ( H_s^0, H_b^0, U_s^0, U_b^0)\|_{H^{s}}.\]

	Moreover, the maximal existence time (resp. the emerging solution in $C([0,T^\star);H^s(\R)^4)$) is a lower semi-continuous (resp. continuous) function of the initial data in $H^s(\R)^4$ and if $T^\star<\infty$ then
	\[ \| ( H_s(t,\cdot), H_b(t,\cdot), U_s(t,\cdot), U_b(t,\cdot))\|_{H^{s_0}}\to \infty \text{ as } t\to T^\star.\]
\end{Prop}
\begin{proof}
	The proof has been given in~\cite{Adim}, but we sketch it here for convenience. We view~\eqref{eq.SV2-kappa} as a system of two transport-diffusion equations and two transport equations, coupled only through order-zero source terms:
	\[
	\begin{cases}
		\partial_t H_s + (\barU_s+U_s)\partial_x H_s - \kappa\partial_x^2 H_s =-(\barH_s+H_s)\partial_x U_s, \\
		\partial_t H_b + (\barU_b+U_b)\partial_x H_b - \kappa\partial_x^2 H_b =-(\barH_b+H_b)\partial_x U_b, \\
		\partial_t U_s + (\barU_s+U_s-\kappa\frac{\partial_x H_s}{\barH_s+H_s})\partial_x U_s = -  \partial_x H_s -  \partial_x H_b,\\
		\partial_t U_b + (\barU_b+U_b-\kappa\frac{\partial_x H_b}{\barH_b+H_b})\partial_x U_b = - \frac{\rho_s}{\rho_b}\partial_x H_s -  \partial_x H_b.
	\end{cases}
	\]
	The standard theory on transport and transport-diffusion equations (see~\cite{BahouriCheminDanchin11}) allows to bootstrap the standard fixed-point strategy through Picard iterates
	\[
	\begin{cases}
		\partial_t H_s^{n+1} + (\barU_s+U_s^{n})\partial_x H_s^{n+1} - \kappa\partial_x^2 H_s^{n+1} = -(\barH_s+H_s^{n}) \partial_x U_s^{n}, \\
		\partial_t H_b^{n+1} + (\barU_b+U_b^{n})\partial_x H_b^{n+1} - \kappa\partial_x^2 H_b^{n+1} = -(\barH_b+H_b^{n}) \partial_x U_b^{n}, \\
		\partial_t U_s^{n+1} + (\barU_s+U_s^{n}-\kappa\frac{\partial_x H_s^{n}}{\barH_s+H_s^{n}})\partial_x U_s^{n+1} = -  \partial_x H_s^{n} -  \partial_x H_b^{n},\\
		\partial_t U_b^{n+1} + (\barU_b+U_b^{n}-\kappa\frac{\partial_x H_b^{n}}{\barH_b+H_b^{n}})\partial_x U_b^{n+1} = - \frac{\rho_s}{\rho_b}\partial_x H_s^{n} -  \partial_x H_b^{n},
	\end{cases}
	\]
	which defines a sequence satisfying the following estimates (where $c_0$ is a non-essential constant depending on $s$)
	\begin{multline*}\max(\{\|(H_s^{n+1},H_b^{n+1},U_s^{n+1},U_b^{n+1})\|_{L^\infty(0,t;H^s)},\kappa^{1/2} \|(\partial_xH_s^{n+1},\partial_xH_b^{n+1})\|_{L^2(0,t;H^s)}\})\\
		\leq \Big(\|(H_s^{0},H_b^{0})\|_{H^s}
		+\kappa^{-1/2} \big((\barH_s +\| H_s^{n}\|_{L^\infty(0,t;H^s)})\| U_s^{n}\|_{L^2(0,t;H^s)}  
			+(\barH_b +\| H_b^{n}\|_{L^\infty(0,t;H^s)})  \| U_b^{n}\|_{L^2(0,t;H^s)} \big)\\
		+ \|(U_s^{0},U_b^{0})\|_{H^s}+(1+\frac{\rho_s}{\rho_b}) \|(\partial_x H_s^{n} ,  \partial_x H_b^{n})\|_{L^1(0,t;H^s)}\Big)\\
		\times \exp\big(c_0\|(U_s^{n},U_b^{n})\|_{L^1(0,t;H^s)}+c_0\kappa \|(\tfrac{\partial_x H_s^{n}}{\barH_s+H_s^{n}},\tfrac{\partial_x H_b^{n}}{\barH_b+H_b^{n}})\|_{L^1(0,t;H^s)}\big)
	\end{multline*}
	and converging in $C(0,t;H^s)$ provided $t\in(0, \kappa T]$ where $T$ is chosen sufficiently small.
	
	The proof of the continuity of the flow map can be obtained along the same lines, using the continuity with respect to the initial data and Lipschitz-continuity with respect to source terms of the transport-diffusion and transport equations.
\end{proof}

The terminology ``Small time well-posedness'' in Proposition~\ref{P.WP-small-time-SV2} refers to the fact that the time of existence and control of solutions of the above result is limited to $T^\star\gtrsim \kappa$, and in particular may vanish as $\kappa\searrow 0$. Notice that, differently from the statement of Proposition~\ref{P.WP-SV}, we do not assume that the flow is stably stratified, namely $\rho_s<\rho_b$.
Assuming additionally that the flow is stably stratified, the first author improved this result in situations with small shear velocities and small deviations from the shear equilibrium, and obtains in~\cite{Adim} the existence and uniform control of solution up to times ${T^\star\gtrsim \big(1+\kappa^{-1}(|\barU_b-\barU_s|^2+M_0^2)\big)^{-1}}$.
	
	In the following results, we complete the picture by showing that, in the situation where the shear velocity is small enough to guarantee that the flow is in the hyperbolic domain of the non-diffusive equation, then the time of existence is uniform with respect to $\kappa\in(0,1]$. In fact we shall prove the expected property that solutions to the diffusive system~\eqref{eq.SV2-kappa} converge as $\kappa\searrow 0$ towards corresponding solutions to the non-diffusive system~\eqref{eq.SV2} as long as the non-diffusive solution is bounded.

In order to obtain stability estimates that are uniform with respect to $\kappa\in(0,1]$, we rely on two main ideas. Firstly, we shall use energy estimates using the explicit symmetrizer adapted to the non-diffusive system introduced in Lemma~\ref{L.symmetrizer} (while the strategy in~\cite{Adim} used only its block-diagonal component). Because the non-diagonal components of the symmetrizer behave poorly with respect to the diffusive contributions, we need another ingredient. Specifically, we notice that the total velocities $V_\ell:=U_\ell-\kappa\frac{\partial_x H_\ell}{\barH_\ell+H_\ell}$ ($\ell\in\{s,b\}$) associated to solutions to~\eqref{eq.SV2-kappa} satisfy the system
\begin{equation}
\begin{cases}\label{eq.SV2-BD}
	\partial_t H_s + \partial_x ((\barH_s+H_s)(\barU_s+V_s)) =0, \\
	\partial_t H_b + \partial_x ((\barH_b+H_b)(\barU_s+V_b)) =0, \\
	\partial_t V_s + (\barU_s+V_s-\kappa\frac{\partial_x H_s}{\barH_s+H_s})\partial_x V_s +  \partial_x H_s +  \partial_x H_b=\kappa\partial_x^2 V_s,\\
	\partial_t V_b + (\barU_b+V_b-\kappa\frac{\partial_x H_b}{\barH_b+H_b})\partial_x V_b + \frac{\rho_s}{\rho_b}\partial_x H_s +  \partial_x H_b=\kappa\partial_x^2 V_b.
\end{cases}
\end{equation}
We observe that diffusive terms act as effective viscosity contributions on the total velocities. The last two equations read equivalently
\[
\begin{cases}
	\partial_t V_s + (\barU_s+V_s)\partial_x V_s +  \partial_x H_s +  \partial_x H_b=\frac{\kappa}{\barH_s+H_s}\partial_x((\barH_s+H_s)\partial_x V_s),\\
	\partial_t V_b + (\barU_b+V_b)\partial_x V_b + \frac{\rho_s}{\rho_b}\partial_x H_s +  \partial_x H_b=\frac{\kappa}{\barH_b+H_b}\partial_x((\barH_b+H_b)\partial_x V_b),
\end{cases}
\]
and we recognize the shallow-water equations with degenerate viscosity contributions which were advocated by Gent in~\cite{Gent93} and derived from the Navier--Stokes equations in~\cite{GerbeauPerthame01,BreschNoble07}. In their analysis of such systems (and generalizations thereof), Bresch and Desjardins introduced the so-called BD entropy in~\cite{BreschDesjardinsLin03,BreschDesjardins03,BreschDesjardins04} (see also~\cite{BreschDesjardinsZatorska15} for a refined analysis), which is based precisely in the reformulation of~\eqref{eq.SV2-BD} as~\eqref{eq.SV2-kappa} (in dimension $d=1$). 

In the same spirit, we combine the regularizing effects of the effective diffusivity and viscosity terms with aforementioned energy estimates, which allows us to obtain suitable stability estimates presented in Lemma~\ref{L.estimate-linearized}, below.

Applying Lemma~\ref{L.estimate-linearized} to (the derivatives of) solutions to~\eqref{eq.SV2-kappa}-\eqref{eq.SV2-BD}, we find a time of existence which is uniform with respect to $\kappa$. We state the result in forthcoming Proposition~\ref{P.uniform-WP}.

Applying Lemma~\ref{L.estimate-linearized} to (the derivatives of) the difference between solutions to~\eqref{eq.SV2-kappa}-\eqref{eq.SV2-BD} and corresponding solutions to the non-diffusive system~\eqref{eq.SV2} ($\kappa=0$) yields the aforementioned convergence of the former towards the latter as $\kappa\searrow 0$, on a time interval defined by the solutions without diffusivity. We state the result in forthcoming Proposition~\ref{P.convergence-SV2}.

\begin{Lem}[Stability]\label{L.estimate-linearized}
	Let $\varsigma\in(0,1)$ and $M>0$. There exists $c>0$ depending only on $\varsigma$ and $C>0$ depending also on $M$
	such that the following holds. 
	
	Let  $\kappa\in(0,1]$, $0<\rho_s<\rho_b$ and $\bU:=(H_s,H_b,U_s,U_b),\ \bV:=(H_s,H_b,V_s,V_b)
		\in {C([0,T]; W^{1,\infty}(\R)^4)}\cap  {C^1([0,T];L^{\infty}(\R)^4)}$  be such that for all $t\in [0,T]$,
		the hyperbolicity condition holds:
		\[\forall x\in\R,\quad  ( \rho_s,\rho_b,H_s(t,x), H_b(t,x), U_s(t,x), U_b(t,x)) \in \mathfrak p^\varsigma,
		\]
		where $\mathfrak p^\varsigma$ is defined in~\eqref{eq.hyperbolicity-condition}, and
	\[ \|  \bU (t,\cdot)\|_{W^{1,\infty}}+\|  \bV(t,\cdot) \|_{W^{1,\infty}}  +\kappa \|(\partial_x^2 H_s(t,\cdot),\partial_x^2 H_b(t,\cdot))\|_{L^\infty}+ \kappa^{-1} \| (\bU-\bV)(t,\cdot)\|_{L^\infty}+ \| \partial_t \bU (t,\cdot)\|_{L^\infty} \leq M.
	\]
	
	Let $\dot \bU:=(\dot H_s,\dot H_b,\dot U_s,\dot U_b)$ and $\dot \bV:=(\dot H_s,\dot H_b,\dot V_s,\dot V_b)$ be sufficiently regular solutions to the linearized equations with remainders
	\begin{align*}
		\partial_t \dot \bU+\A^\kappa(\bU)\partial_x \dot \bU=\kappa \D_1 \partial_x^2 \dot \bU+\bR_{\bU},\\
		\partial_t \dot \bV+\A^\kappa(\bV)\partial_x \dot \bV=\kappa \D_2 \partial_x^2 \dot \bV+\bR_{\bV},
	\end{align*}
	where we denote
	\[ \A^\kappa:(H_s,H_b,U_s,U_b)\mapsto\begin{pmatrix}
		U_s & 0 & H_s & 0 \\
		0& U_b & 0 & H_b\\
		1 & 1 & U_s-\kappa\tfrac{\partial_x H_s}{H_s} & 0 \\
		\frac{\rho_s}{\rho_b}  & 1 &0 & U_b-\kappa\tfrac{\partial_x H_b}{H_b} 
	\end{pmatrix},\]
	and
	\[ \D_1 =\begin{pmatrix}
		1 & 0 & 0 & 0 \\
		0& 1 & 0 & 0\\
		0&0&0 & 0 \\
		0&0&0 & 0 
	\end{pmatrix}, \qquad 
	\D_2 =\begin{pmatrix}
		0 & 0 & 0 & 0 \\
		0& 0 & 0 & 0\\
		0&0&1 & 0 \\
		0&0&0 & 1 
	\end{pmatrix}. \]
	Moreover, denote $R:=(R_s,R_b)$ such that
	\begin{equation}\label{eq.rell}
	\dot V_s = \dot U_s-\kappa\frac{\partial_x \dot H_s}{H_s} + R_s, \quad \dot V_b = \dot U_b-\kappa\frac{\partial_x \dot H_b}{H_b} + R_b. 
	\end{equation}
	Then, for any $t\in [0,T]$, one has the estimate
	\begin{align*}
		 &\| \dot \bU(t,\cdot)\|_{L^{2}}+ \| \dot \bV(t,\cdot)\|_{L^{2}} + c  \kappa^{1/2} \| \partial_x \dot \bV(t,\cdot)\|_{L^2(0,t;L^{2})} 
		  \leq c^{-1}\, \big(\| \dot \bU(t=0,\cdot)\|_{L^{2}}+ \| \dot \bV(t=0,\cdot)\|_{L^{2}}\big) \exp(C\, M\, t)\\
		 &\quad +C \int_0^t \big(\| \bR_{\bU}(t',\cdot)\|_{L^2}+\| \bR_{\bV}(t',\cdot)\|_{L^2}+M\|R(t',\cdot)\|_{H^1}\big) \exp(C\, M\, (t-t')) \dd t.
	\end{align*}
\end{Lem}
\begin{proof}
	Denote
	\[\S^\lambda:(H_s,H_b,U_s,U_b)\mapsto \begin{pmatrix}
		\frac{\rho_s}{\rho_b} & \frac{\rho_s}{\rho_b} & \frac{\rho_s}{\rho_b} U_s^\lambda & 0 \\
		\frac{\rho_s}{\rho_b} & 1 & 0 & U_b^\lambda\\
		\frac{\rho_s}{\rho_b} U_s^\lambda & 0 & \frac{\rho_s}{\rho_b} H_s & 0 \\
		0 & U_b^\lambda  & 0 &  H_b
	\end{pmatrix}\]
	where $U_\ell^\lambda := U_\ell -\lambda$ (for $\ell\in\{s,b\}$) with $\lambda$ provided in Lemma~\ref{L.symmetrizer}.
	Using that $\S^\lambda(\cdot)$ and $\S^\lambda(\cdot)\A^0(\cdot)$ are symmetric, and integration by parts, we have the energy identities
	\begin{align*}&\frac12\frac{\dd}{\dd t} \Big( \big( \S^\lambda(\bU) \dot \bU, \dot \bU\big)_{L^2} +  \big( \S^\lambda(\bU) \dot \bV,\dot \bV\big)_{L^2} \Big) \\
		& = \big( \S^\lambda(\bU) \partial_t \dot \bU,\dot \bU\big)_{L^2} +  \big( \S^\lambda(\bU) \partial_t \dot \bV,\dot \bV\big)_{L^2} + \frac12 \big( [\partial_t,\S^\lambda(\bU)] \dot \bU,\dot \bU\big)_{L^2} + \frac12 \big( [\partial_t,\S^\lambda(\bV) ]\dot \bV,\dot \bV\big)_{L^2}\\
		& = -\big( \S^\lambda(\bU) \A^\kappa(\bU)\partial_x \dot \bU,\dot \bU\big)_{L^2} -  \big( \S^\lambda(\bU)\A^\kappa(\bV)\partial_x \dot \bV,\dot \bV\big)_{L^2} + \kappa\big( \S^\lambda(\bU) \D_1\partial_x^2 \dot \bU,\dot \bU\big)_{L^2}+ \kappa\big( \S^\lambda(\bU) \D_2\partial_x^2 \dot \bV,\dot \bV\big)_{L^2}  \\
		&\qquad + \big( \S^\lambda(\bU) \bR_{\bU},\dot \bU\big)_{L^2}+ \big( \S^\lambda(\bU) \bR_{\bV},\dot \bV\big)_{L^2} +\frac12 \big( [\partial_t,\S^\lambda(\bU)] \dot \bU,\dot \bU\big)_{L^2} + \frac12 \big( [\partial_t,\S^\lambda(\bU) ]\dot \bV,\dot \bV\big)_{L^2}\\
		& = \frac12 \big( [\partial_x,\S^\lambda(\bU)\A^0(\bU)] \dot \bU,\dot \bU\big)_{L^2} + \frac12 \big( [\partial_x,\S^\lambda(\bV)\A^0(\bV) ]\dot \bV,\dot \bV\big)_{L^2}\\
		&\qquad - \big( \S^\lambda(\bU)(\A^\kappa(\bU)-\A^0(\bU))\partial_x \dot \bU,\dot \bU\big)_{L^2} - \big( \S^\lambda(\bV)(\A^\kappa(\bV)-\A^0(\bV))\partial_x\dot \bV,\dot \bV\big)_{L^2}\\
		&\qquad -  \big( (\S^\lambda(\bU)-\S^\lambda(\bV))\A^\kappa(\bV)\partial_x \dot \bV,\dot \bV\big)_{L^2} 
		+\frac12 \big( [\partial_t,\S^\lambda(\bU)] \dot \bU,\dot \bU\big)_{L^2} + \frac12 \big( [\partial_t,\S^\lambda(\bU) ]\dot \bV,\dot \bV\big)_{L^2}\\
		&\qquad + \big( \S^\lambda(\bU) \bR_{\bU},\dot \bU\big)_{L^2}+ \big( \S^\lambda(\bU) \bR_{\bV},\dot \bV\big)_{L^2}
		 + \kappa\big( \S^\lambda(\bU) \D_1\partial_x^2 \dot \bU,\dot \bU\big)_{L^2}+ \kappa\big( \S^\lambda(\bU) \D_2\partial_x^2 \dot \bV,\dot \bV\big)_{L^2} \\
		 &=:A- \big( \S^\lambda(\bU)(\A^\kappa(\bU)-\A^0(\bU))\partial_x \dot \bU,\dot \bU\big)_{L^2}+\kappa\big( \S^\lambda(\bU) \D_1\partial_x^2 \dot \bU,\dot \bU\big)_{L^2}+ \kappa\big( \S^\lambda(\bU) \D_2\partial_x^2 \dot \bV,\dot \bV\big)_{L^2}. 
	\end{align*}
	By means of Cauchy--Schwarz inequality
	 we find that 
	\[|A|\leq C\,  \big( \|\dot \bU\|_{L^{2}}+ \| \dot \bV\|_{L^{2}}\big)\times \big( M  \|\dot \bU\|_{L^{2}}+ M\| \dot \bV\|_{L^{2}} + \kappa M \|\partial_x\dot \bV\|_{L^2} + \| \bR_{\bU}\|_{L^2}+\| \bR_{\bV}\|_{L^2}\big) ,
	\]
	where $C$ denotes a multiplicative constant depending only on $\varsigma$ and $M$, and which may change from line to line.
	We now focus on the remaining terms. We first notice that defects of symmetry in $\S^\lambda(\bU)(\A^\kappa(\bU)-\A^0(\bU))$ arise only in the first two rows, and that the first two components of $\dot\bU$ equal the first two components of $\dot\bV$. Hence using integration by parts and Cauchy--Schwarz inequality we infer
	\[|\big( \S^\lambda(\bU)(\A^\kappa(\bU)-\A^0(\bU))\partial_x \dot \bU,\dot \bU\big)_{L^2}|\leq \kappa\,  C\, M\, \|\dot \bU\|_{L^{2}}\times \big( \|\dot \bU\|_{L^2} +\|\partial_x\dot \bV\|_{L^2}\big).\]
	 Then, again making use of the identity $\D_1\partial_x^2 \dot \bU=(\partial_x^2\dot H_1,\partial_x^2\dot H_2,0,0)=\D_1\partial_x^2 \dot \bV$ we infer that
	\[ \kappa\big( \S^\lambda(\bU) \D_1\partial_x^2 \dot \bU,\dot \bU\big)_{L^2}+\kappa\big( \S^\lambda(\bU) \D_2\partial_x^2 \dot \bV,\dot \bV\big)_{L^2}
	= \kappa\big( \S^\lambda(\bU) (\D_1+\D_2)\partial_x^2 \dot \bV,\dot \bV\big)_{L^2}+\kappa\big( \S^\lambda(\bU) \D_1\partial_x^2 \dot \bV,\dot \bU-\dot \bV\big)_{L^2}.\]
	After integration by parts, and since $\D_1+\D_2=\Id$, we find that 
	\[ \kappa\big( \S^\lambda(\bU) (\D_1+\D_2)\partial_x^2 \dot \bV,\dot \bV\big)_{L^2}\leq -\kappa\big( \S^\lambda(\bU) \partial_x \dot \bV, \partial_x \dot \bV\big)_{L^2} + \kappa\,  C\, M\, \| \partial_x \dot \bV\|_{L^2} \| \dot \bV\|_{L^2}.\]
	Then, using from~\eqref{eq.rell} that $\kappa \partial_x \dot H_\ell = H_\ell ((\dot U_\ell - \dot V_\ell)+R_\ell)$ (where $\ell\in\{s,b\}$), we obtain the identities
	\begin{align*}\kappa\big( \S^\lambda(\bU) \D_1\partial_x^2 \dot \bV,\dot \bU-\dot \bV\big)_{L^2}
		&=\kappa\sum_{\ell\in\{s,b\}} \big( \rho_\ell U_\ell^\lambda\partial_x^2 \dot H_\ell,\dot U_\ell-\dot V_\ell\big)_{L^2}\\
		&=\sum_{\ell\in\{s,b\}} \big( \rho_\ell U_\ell^\lambda\partial_x ( H_\ell(\dot U_\ell-\dot V_\ell+R_\ell)),\dot U_\ell-\dot V_\ell\big)_{L^2}\\
		&=\sum_{\ell\in\{s,b\}} \frac{\rho_\ell}2\big( ( U_\ell^\lambda \partial_x H_\ell-H_\ell\partial_x U_\ell^\lambda )(\dot U_\ell-\dot V_\ell),\dot U_\ell-\dot V_\ell\big)_{L^2}\\
		&\quad
		 + \kappa\sum_{\ell\in\{s,b\}}\rho_\ell \big(  U_\ell^\lambda\partial_x(H_\ell R_\ell) ,\dot U_\ell-\dot V_\ell\big)_{L^2},
	\end{align*}
	where we used integration by parts in the last line. We infer
	\[ \kappa\big( \S^\lambda(\bU) \D_1\partial_x^2 \dot \bV,\dot \bU-\dot \bV\big)_{L^2}\leq \, C \, M\, \big(\|  \dot \bU\|_{L^2}+ \|  \dot \bV\|_{L^2} \big)^2+ C\, M\, \big(\|  \dot \bU\|_{L^2}+ \|  \dot \bV\|_{L^2} \big) \| R\|_{H^1} .\]

	Combining all these estimate and denoting 
	\[\cE:=   \big( \S^\lambda(\bU) \dot \bU, \dot \bU\big)_{L^2} +  \big( \S^\lambda(\bU) \dot \bV,\dot \bV\big)_{L^2} , \]
	one has
	\begin{multline} \frac12\frac{\dd}{\dd t}\cE + \kappa\big( \S^\lambda(\bU) \partial_x \dot \bV, \partial_x \dot \bV\big)_{L^2}
		\leq C \, M\, \big( \| \dot \bU\|_{L^{2}}+ \| \dot \bV\|_{L^{2}} + \kappa \| \partial_x \dot \bV\|_{L^{2}}\big) \big( \|\dot \bU\|_{L^{2}}+ \| \dot \bV\|_{L^{2}}\big) \\
		 + C \big(\| \bR_{\bU}\|_{L^2}+\| \bR_{\bV}\|_{L^2}+M\|R\|_{H^1}\big) \big( \|\dot \bU\|_{L^{2}}+ \| \dot \bV\|_{L^{2}}\big).
	\end{multline}
	By using that $\S^\lambda(\bU)$ is definite positive, we find that there exists $c>0$ depending only on $\varsigma$ such that
	\[\cE\ge c^2 \| \dot \bU\|_{L^{2}}^2+ c^2 \| \dot \bV\|_{L^{2}}^2,  \quad  \big( \S^\lambda(\bU) \partial_x \dot V,\partial_x \dot V\big)_{L^2} \geq c^2\| \partial_x \dot \bV\|_{L^{2}}^2.\]
	Hence we find (using the Peter Paul inequality and augmenting $C$) that 
	\[ \frac12\frac{\dd}{\dd t}\cE + \tfrac12 c^2\kappa\| \partial_x \dot \bV\|_{L^{2}}^2
		\leq C \, M\, \cE + C \big(\| \bR_{\bU}\|_{L^2}+\| \bR_{\bV}\|_{L^2}+M\|R\|_{H^1}\big) \, \cE^{1/2},\]
	and the result follows by Gronwall's Lemma.
\end{proof}

\begin{Prop}[Large-time well-posedness]\label{P.uniform-WP}
	Let $s\geq s_0>3/2$, $\varsigma\in(0,1)$ and $M_0>0$. There exists  
	$C>0$ and $T>0$ such that the following holds. 
	
	Let  $\kappa\in(0,1]$,  $(\rho_s,\rho_b, \barH_s,\barH_b,\barU_s,\barU_b)\in\R^6$ such that $\barH_s+\barH_b=1$ and $\barU_s+\barU_b=0$, and let $( H_s^0, H_b^0, U_s^0, U_b^0)\in H^{s+1}(\R)^2\times H^s(\R)^2$ such that 
	the hyperbolicity condition holds:
	\[ \forall x\in\R,\quad ( \rho_s,\rho_b,\barH_s+ H_s^0(x), \barH_b + H_b^0(x), \barU_s+ U_s^0(x),\barU_b+ U_b^0(x)) \in \mathfrak p^\varsigma,
	\]
	where $\mathfrak p^\varsigma$ is defined in~\eqref{eq.hyperbolicity-condition}, and
	\[ \| ( H_s^0, H_b^0, U_s^0, U_b^0,\kappa\partial_x H_s^0,\kappa\partial_x H_b^0)\|_{H^{s_0}}\leq M_0.\]
	Denote $( H_s,  H_b, U_s, U_b)\in C([0,T^\star);H^s(\R)^4)$ the maximal-in-time solution to~\eqref{eq.SV2-kappa} emerging from the initial data $(H_s,H_b,U_s,U_b)\big\vert_{t=0}=( H_s^0,  H_b^0,  U_s^0, U_b^0)$ as defined in Proposition~\ref{P.WP-small-time-SV2}. 
	
	One has $T^\star>T/M_0$ and for any $t\in[0,T/M_0]$,
	 	\[ \forall x\in\R,\quad  ( \rho_s,\rho_b,\barH_s+ H_s(t,x), \barH_b + H_b(t,x), \barU_s+ U_s(t,x),\barU_b+ U_b(t,x)) \in \mathfrak p^{\varsigma/2},\]
	 	and
	 \begin{multline*}\| (H_s,  H_b, U_s, U_b, \kappa\partial_x H_s, \kappa\partial_x H_b)(t,\cdot)\|_{H^s}+\| (\partial_t H_s,\partial_t H_b,\partial_t U_s,\partial_t U_b)(t,\cdot)\|_{H^{s-1}} \\\leq C \| ( H_s^0, H_b^0, U_s^0, U_b^0,\kappa\partial_x H_s^0,\kappa\partial_x H_b^0)\|_{H^{s}}.
	 \end{multline*}
\end{Prop}
\begin{proof}
	We assume that the initial data is smooth, so that $(H_s,H_b,U_s,U_b,V_s,V_b)$ are smooth on their domain of existence. The general case is obtained by regularizing the initial data and passing to the limit, thanks to the persistence of regularity and continuity of the flow map stated in Proposition~\ref{P.WP-small-time-SV2}.
	
	Denote $ V_\ell:= U_\ell-\kappa\frac{\partial_x  H_\ell}{\barH_\ell+H_\ell}$ ($\ell\in\{s,b\}$), $\Lambda^s:=(\Id-\partial_x^2)^{s/2}$ and
	\[(\dot H_s,\dot H_b,\dot U_s,\dot U_b,\dot V_s,\dot V_b):=(\Lambda^s H_s,\Lambda^s H_b,\Lambda^s U_s,\Lambda^s U_b, \Lambda^s V_s,\Lambda^s V_b).\]
	Applying the operator $\Lambda^s$ to~\eqref{eq.SV2-kappa} and~\eqref{eq.SV2-BD}, we obtain
	\[
	\begin{cases}
		\partial_t \dot H_s + (\barH_s+H_s)\partial_x \dot U_s + (\barU_s+U_s)\partial_x \dot H_s =\kappa\partial_x^2 \dot H_s + R_H(H_s,U_s), \\
		\partial_t \dot H_b + (\barH_b+H_b)\partial_x \dot U_b + (\barU_b+U_b)\partial_x \dot H_b =\kappa\partial_x^2 \dot H_b+ R_H(H_b,U_b), \\
		\partial_t \dot U_s + (\barU_s+U_s-\kappa\frac{\partial_x H_s}{\barH_s+H_s})\partial_x \dot U_s +  \partial_x \dot H_s +  \partial_x \dot H_b= R_U(H_s,U_s),\\
		\partial_t \dot U_b + (\barU_b+U_b-\kappa\frac{\partial_x H_b}{\barH_b+H_b})\partial_x \dot U_b + \frac{\rho_s}{\rho_b}\partial_x \dot H_s +  \partial_x \dot H_b=R_U(H_b,U_b),
	\end{cases}
	\]
	and
	\[
	\begin{cases}
		\partial_t \dot H_s + (\barH_s+H_s)\partial_x \dot V_s + (\barU_s+V_s)\partial_x \dot H_s = R_H(H_s,V_s), \\
		\partial_t \dot H_b + (\barH_b+H_b)\partial_x \dot V_b + (\barU_b+V_b)\partial_x \dot H_b = R_H(H_b,V_b), \\
		\partial_t \dot V_s + (\barU_s+V_s-\kappa\frac{\partial_x H_s}{\barH_s+H_s})\partial_x \dot V_s +  \partial_x \dot H_s +  \partial_x \dot H_b=\kappa\partial_x^2\dot V_s+R_U(H_s,V_s),\\
		\partial_t \dot V_b + (\barU_b+V_b-\kappa\frac{\partial_x H_b}{\barH_b+H_b})\partial_x \dot V_b + \frac{\rho_s}{\rho_b}\partial_x \dot H_s +  \partial_x \dot H_b=\kappa\partial_x^2\dot V_s+R_U(H_b,V_b),
	\end{cases}
	\]
	with remainders $R_H$ and $R_V$ defined as
	\[ 	R_H(H,V) := -[\Lambda^s,H]\partial_x V-[\Lambda^s,V]\partial_x H \quad \text{ and } \quad 
	 	R_U(H,V) := -[\Lambda^s,V-\kappa\tfrac{\partial_x H}{\barH+H}]\partial_x V.\]
	Moreover,  applying the operator $\Lambda^s$ to the identity $V_\ell=U_\ell-\kappa\frac{\partial_x H_\ell}{\barH_\ell+H_\ell}$ ($\ell\in\{s,b\}$) yields
	\[ \dot V_\ell = \dot U_\ell-\kappa\frac{\partial_x \dot H_\ell}{\barH_\ell+H_\ell} + R(H_\ell), \]
	with
	\[ R(H)=-\kappa [\Lambda^s,\tfrac{1}{\barH+H}]\partial_x H.\]
	Standard commutator estimates and composition estimates in Sobolev spaces; see {\em e.g.}~\cite[Appendix~B]{Lannes} yield
	\begin{align*} 
		\|R_H(H,V)\|_{L^2}&\leq C \big( \| H\|_{H^{s_0}}+\| V\|_{H^{s_0}}\big)\big( \| H\|_{H^{s}}+\| V\|_{H^{s}}\big), \\
		\|R_U(H,V)\|_{L^2}&\leq C\big( \kappa\| \partial_xH\|_{H^{s_0}}+\| V\|_{H^{s_0}}\big)\big( \kappa\| \partial_x H\|_{H^{s}}+\| V\|_{H^{s}}\big),\\
		\|R(H)\|_{H^1}&\leq C \big(\| H\|_{H^{s_0}}+\kappa\| \partial_x  H\|_{H^{s_0}}\big) \big(\|  H\|_{H^{s}}+\kappa\| \partial_x H\|_{H^{s}}\big) ,
	\end{align*}
	where $C$ is a positive constant depending only on $s,s_0$, $\| H\|_{H^{s_0}}$ and $\inf_{\R}(\barH+H)>0$. 
	
	Moreover, notice that by using the equations~\eqref{eq.SV2-kappa},~\eqref{eq.SV2-BD} and the identity $\kappa\partial_x H_\ell =(\barH_\ell+H_\ell)(U_\ell-V_\ell)$ for $\ell\in\{s,b\}$, we have
	\begin{multline*}  \|  (\partial_t H_s,\partial_t H_b,\partial_t U_s,\partial_t U_b) \|_{H^{s-1}}+ \kappa\|(\partial_x  H_s,\partial_x  H_b)\|_{H^{s}} +\kappa^{-1}\|( U_s- V_s, U_b- V_b)\|_{H^{s-1}}\\ \leq C\, \|  ( H_s, H_b, U_s, U_b, V_s, V_b) \|_{H^{s}} ,
	\end{multline*}
	where the multiplicative constant $C$ depends on $\|  ( H_s, H_b, U_s, U_b, V_s, V_b) \|_{H^{s_0}}$ and $\inf_{\R}(\barH_\ell+ H_\ell)>0$ (for $\ell\in\{s,b\}$).
	
	We may thus apply Lemma~\ref{L.estimate-linearized}, and infer that we can set  $C$  depending on $s_0,M_0,\varsigma$, and $c$ depending only on $\varsigma$, so that as long as
		\begin{equation} \label{eq.H1}
		\forall x\in\R,\quad 	 ( \rho_s,\rho_b,\barH_s+H_s(t,x), \barH_b+H_b(t,x), \barU_s+U_s(t,x), \barU_b+U_b(t,x)) \in \mathfrak p^{\varsigma/2},
	\end{equation}
	and
		\begin{equation} \label{eq.H2}
			\|  ( H_s, H_b, U_s, U_b, V_s, V_b)(t,\cdot) \|_{H^{s_0}} 
	\leq 2c M_0
\end{equation}
	one has (using the standard $H^{s_0}\subset W^{1,\infty}$ continuous embedding and the fact that $0\leq \barH_s,\barH_b\leq 1$ and $\barU_s+\barU_b=0$)
	\begin{multline*}
	 \|  ( H_s, H_b, U_s, U_b, V_s, V_b) (t,\cdot)\|_{H^{s_0}}
	\leq c\, M_0 \exp(C\, M_0\, t)\\
	+C M_0\int_0^t \|  ( H_s, H_b, U_s, U_b, V_s, V_b)(t',\cdot) \|_{H^{s_0}}  \exp(C\, M_0\, (t-t')) \dd t.
\end{multline*}
	Applying Gronwall's Lemma, we find that if $C M_0 t$ is smaller than a universal constant, then 
	\[\| ( H_s, H_b, U_s, U_b, V_s, V_b)(t,\cdot) \|_{H^{s}}
	\leq \frac{3c}2 \, M_0.\]
	What is more  we have (augmenting $C$ if necessary)
	\[ |( H_s(t,\cdot)- H_s^0, H_b(t,\cdot)- H_b^0, U_s(t,\cdot)- U_s^0, U_b(t,\cdot)- U_b^0)|\leq \int_0^t \|(\partial_t  H_s,\partial_t  H_b,\partial_t  U_s,\partial_t  U_b)\|_{L^\infty}\leq C M_0 t.\] 
	 Hence lowering further   $C M_0 t$, we infer that $ ( \rho_s,\rho_b,H_s(t,x), H_b(t,x), U_s(t,x), U_b(t,x)) \in \mathfrak p^{\varsigma/4}$ for all $x\in\R$. By the usual continuity argument we infer that the assumptions~\eqref{eq.H1} and~\eqref{eq.H2} do hold for $t\in[0,T/M_0]$ with $T$ depending on $s,s_0,M_0,\varsigma$. This yields the lower bound on the maximal time of existence and the claimed upper bound on the solution follows from the above estimates replacing $s_0$ with $s$.
\end{proof}

Combined together, Proposition~\ref{P.WP-small-time-SV2} and Proposition~\ref{P.uniform-WP} yield a
time of existence for solutions to~\eqref{eq.SV2-kappa} emerging from sufficiently regular initial data which is independent of $\kappa\in(0,1]$. The following result describes the behavior of these solutions as $\kappa\searrow 0$.

\begin{Prop}[Convergence]\label{P.convergence-SV2}
	Let $s\geq s_0>3/2$, $\varsigma\in(0,1)$ and $M_0>0$. 
	
	Let $\kappa\in(0,1]$, $(\rho_s,\rho_b, \barH_s,\barH_b,\barU_s,\barU_b)\in\R^6$ such that $\barH_s+\barH_b=1$ and $\barU_s+\barU_b=0$, and $( H_s^0, H_b^0, U_s^0, U_b^0)\in H^{s+2}(\R)^4$ be such that 
	the hyperbolicity condition holds
	\[ \forall x\in\R,\quad ( \rho_s,\rho_b,\barH_s+ H_s^0(x), \barH_b + H_b^0(x), \barU_s+ U_s^0(x),\barU_b+ U_b^0(x)) \in \mathfrak p^\varsigma,
	\]
	where $\mathfrak p^\varsigma$ is defined in~\eqref{eq.hyperbolicity-condition}, and
	\[ \| ( H_s^0, H_b^0, U_s^0, U_b^0,\kappa\partial_x H_s^0,\kappa\partial_x H_b^0)\|_{H^{s_0}}\leq M_0.\]
	Denote 
	\begin{itemize}
		\item $(H_s,H_b,U_s,U_b)$ the maximal solution to the non-diffusive system~\eqref{eq.SV2} emerging from the initial data \[(H_s,H_b,U_s,U_b)\big\vert_{t=0}=( H_s^0, H_b^0,  U_s^0,U_b^0)\]
		 as defined in Proposition~\ref{P.WP-SV};
		\item for all $\kappa>0$, $(H_s^\kappa,H_b^\kappa,U_s^\kappa,U_b^\kappa)$ the maximal solution to system~\eqref{eq.SV2-kappa} emerging from the initial data 
		\[(H_s^\kappa,H_b^\kappa,U_s^\kappa,U_b^\kappa)\big\vert_{t=0}=( H_s^0, H_b^0,  U_s^0,U_b^0)\]
		 as defined in Proposition~\ref{P.WP-small-time-SV2};
		\item $T_0>0$ and $c_0>1$ such that for all $t\in[0,T_0]$,
		\[ \forall x\in\R,\quad  ( \rho_s,\rho_b,\barH_s+H_s(t,x), \barH_b+H_b(t,x),\barU_s+U_s(t,x),\barU_b+U_b(t,x)) \in \mathfrak p^{\varsigma/c_0}\]
		and
		\[\| ( H_s, H_b, U_s, U_b)(t,\cdot)\|_{H^{s+2}}\leq c_0 M_0.\] 

	\end{itemize} 
	Then there exists $\kappa_0>0$ and $C>0$, both depending only on $s,s_0,\varsigma,M_0,T_0,c_0$, such that for all $\kappa \in(0,\kappa_0]$, $(H_s^\kappa,H_b^\kappa,U_s^\kappa,U_b^\kappa)(t,\cdot)$ is well-defined for all $t\in[0,T_0]$ and satisfies the hyperbolicity condition
			\[ \forall x\in\R,\quad  ( \rho_s,\rho_b,\barH_s+H_s^\kappa(t,x), \barH_b+H_b^\kappa(t,x),\barU_s+U_s^\kappa(t,x),\barU_b+U_b^\kappa(t,x)) \in \mathfrak p^{\varsigma/(2c_0)}\]
			and the upper bound 
	\[\| ( H_s^\kappa, H_b^\kappa, U_s^\kappa, U_b^\kappa)(t,\cdot)\|_{H^{s}} \leq 2 c_0 M_0.\]
		Moreover, one has for all $t\in[0,T_0]$
		\[\| ( H_s^\kappa- H_s, H_b^\kappa- H_b, U_s^\kappa- U_s, U_b^\kappa- U_b)(t,\cdot)\|_{H^{s}} \leq \kappa\, C M_0.\] 
\end{Prop}
\begin{proof}
	Denote $V_\ell:=U_\ell$  and $V_\ell^\kappa:=U_\ell^\kappa-\kappa\frac{\partial_x H_\ell^\kappa}{\barH_\ell+H_\ell^\kappa}$ ($\ell\in\{s,b\}$), $\Lambda^{s}:=(\Id-\partial_x^2)^{s/2}$, and
	\[(\dot H_s,\dot H_b,\dot U_s,\dot U_b,\dot V_s,\dot V_b):=(\Lambda^{s}( H_s^\kappa-H_s),\Lambda^{s}( H_b^\kappa- H_b),\Lambda^{s}( U_s^\kappa- U_s),\Lambda^{s}( U_b^\kappa- U_b), \Lambda^{s}( V_s^\kappa- V_s),\Lambda^{s}( V_b^\kappa- V_b)).\]
	Substracting~\eqref{eq.SV2-kappa},~\eqref{eq.SV2-BD} and~\eqref{eq.SV2} we obtain
	\[
	\begin{cases}
		\partial_t \dot H_s + (\barH_s+H_s^\kappa)\partial_x \dot U_s + (\barU_s+U_s^\kappa)\partial_x \dot H_s =\kappa\partial_x^2 \dot H_s +  \kappa \Lambda^{s}\partial_x^2 H_s + R_H(H_s^\kappa,U_s^\kappa,H_s,U_s), \\
		\partial_t \dot H_b + (\barH_b+H_b^\kappa)\partial_x \dot U_b + (\barU_b+U_b^\kappa)\partial_x \dot H_b =\kappa\partial_x^2 \dot H_b+  \kappa \Lambda^{s}\partial_x^2 H_b + R_H(H_b^\kappa,U_b^\kappa,H_b,U_b), \\
		\partial_t \dot U_s + (\barU_s+U_s^\kappa-\kappa\frac{\partial_x H_s^\kappa}{\barH_s+H_s^\kappa})\partial_x \dot U_s +  \partial_x \dot H_s +  \partial_x \dot H_b= \kappa\Lambda^s \big(\frac{\partial_x H_s^\kappa}{\barH_s+H_s^\kappa} \partial_x U_s\big)+ R_U(H_s^\kappa,U_s^\kappa,H_s,U_s),\\
		\partial_t \dot U_b + (\barU_b+U_b^\kappa-\kappa\frac{\partial_x H_b^\kappa}{\barH_b+H_b^\kappa})\partial_x \dot U_b + \frac{\rho_s}{\rho_b}\partial_x \dot H_s +  \partial_x \dot H_b=\kappa\Lambda^s \big(\frac{\partial_x H_b^\kappa}{\barH_b+H_b^\kappa} \partial_x U_b\big) + R_U(H_b^\kappa,U_b^\kappa,H_b,U_b),
	\end{cases}
	\]
	and
	\[
	\begin{cases}
		\partial_t \dot H_s + (\barH_s+H_s^\kappa)H_s^\kappa\partial_x \dot V_s + (\barU_s+V_s^\kappa)\partial_x \dot H_s = R_H(H_s^\kappa,V_s^\kappa,H_s,V_s), \\
		\partial_t \dot H_b + (\barH_b+H_b^\kappa)H_b^\kappa\partial_x \dot V_b + (\barU_b+V_b^\kappa)\partial_x \dot H_b = R_H(H_b^\kappa,V_b^\kappa,H_b,V_b), \\
		\partial_t \dot V_s + (\barU_s+V_s^\kappa-\kappa\frac{\partial_x H_s^\kappa}{\barH_s+H_s^\kappa})\partial_x \dot V_s +  \partial_x \dot H_s +  \partial_x \dot H_b\\
		\hspace{5cm}=\kappa\partial_x^2\dot V_s+  \kappa \Lambda^{s}\partial_x^2 V_s+\kappa\Lambda^s \big(\frac{\partial_x H_s^\kappa}{\barH_s+H_s^\kappa} \partial_x V_s\big)+R_U(H_s^\kappa,V_s^\kappa,H_s,V_s),\\
		\partial_t \dot V_b + (\barU_b+V_b^\kappa-\kappa\frac{\partial_x H_b^\kappa}{\barH_b+H_b^\kappa})\partial_x \dot V_b + \frac{\rho_s}{\rho_b}\partial_x \dot H_s +  \partial_x \dot H_b\\
		\hspace{5cm}=\kappa\partial_x^2\dot V_b+  \kappa \Lambda^{s}\partial_x^2 V_b+\kappa\Lambda^s \big(\frac{\partial_x H_b^\kappa}{\barH_b+H_b^\kappa} \partial_x V_b\big)+R_U(H_b^\kappa,V_b^\kappa,H_b,V_b),
	\end{cases}
	\]
	where, for any $\ell\in\{s,b\}$, we denote
	\begin{align*} 
		R_H(H_\ell^\kappa,U_\ell^\kappa,H_\ell,U_\ell) &= 
			-\Lambda^{s} \big((H_\ell^\kappa-H_\ell)\partial_x  U_\ell + (U_\ell^\kappa-U_\ell)\partial_x H_\ell\big)\\
			&\quad 
			-[\Lambda^{s},H_\ell^\kappa ] (\partial_x U_\ell^\kappa-\partial_x U_\ell)
			-[\Lambda^{s},U_\ell^\kappa ] (\partial_x H_\ell^\kappa-\partial_x H_\ell), \\
		R_U(H_\ell^\kappa,U_\ell^\kappa,H_\ell,U_\ell) &= 
			-\Lambda^{s} \big((U_\ell^\kappa- U_\ell)\partial_x U_\ell\big)
			-[\Lambda^{s},U_\ell^\kappa-\kappa\tfrac{\partial_x H_\ell^\kappa}{\barH_\ell+H_\ell^\kappa} ](\partial_x U_\ell^\kappa -\partial_x U_\ell).
	\end{align*}
	Moreover, one has by definition 
	\[ \dot V_\ell= \dot U_\ell -\kappa\frac{\partial_x \dot H_\ell}{\barH_\ell+H_\ell^\kappa}+ R(H_\ell^\kappa,H_\ell)  \]
	where
	\[  R(H_\ell^\kappa,H_\ell) =-\kappa [\Lambda^{s}, \tfrac1{\barH_\ell+H_\ell^\kappa}]\partial_x  H_\ell^\kappa-\kappa\frac{\partial_x \Lambda^{s}H_\ell}{\barH_\ell+H_\ell^\kappa}.\]
	Standard commutator estimates and composition estimates in Sobolev spaces; see {\em e.g.}~\cite[Appendix~B]{Lannes} yield
	\begin{align*} 
		\|R_H(H_\ell^\kappa,V_\ell^\kappa,H_\ell,V_\ell)\|_{L^2}&\leq C \big( \| \partial_x H_\ell\|_{H^{s}}+\| \partial_x V_\ell\|_{H^{s}}+\| \partial_x H_\ell^\kappa\|_{H^{s-1}}+\| \partial_x V_\ell^\kappa\|_{H^{s-1}}\big)\big( \| \dot H_\ell\|_{L^2}+\| \dot V_\ell\|_{L^2}\big), \\
		\|R_U(H_\ell^\kappa,V_\ell^\kappa,H_\ell,V_\ell)\|_{L^2}&\leq C\big( \| \partial_x V_\ell\|_{H^{s}}+\kappa\| \partial_x H_\ell^\kappa\|_{H^{s}}+\| \partial_x V_\ell^\kappa\|_{H^{s-1}}\big)\| \dot V_\ell\|_{L^2},\\
		\|R(H_\ell^\kappa,H_\ell)\|_{H^1}&\leq C \kappa\big( \| \partial_x H_\ell^\kappa\|_{H^{s}} + \| \partial_x H_\ell\|_{H^{s+1}}\big),
	\end{align*}
	where $C$ is a positive constant depending only on $s$, $s_0$, $\| H_\ell^\kappa\|_{H^{s_0}}$   
	 and 
	$\inf_{\R} \barH_\ell+H_\ell^\kappa>0$ for any $\ell\in\{s,b\}$. 
	
	Moreover notice that by the equations~\eqref{eq.SV2-kappa} and~\eqref{eq.SV2-BD} and using the identity $\kappa\partial_x H_\ell^\kappa =(\barH_\ell+H_\ell^\kappa)(U_\ell^\kappa-V_\ell^\kappa)$ we have
	\begin{multline*}  \|  (\partial_t H_s^\kappa,\partial_t H_b^\kappa,\partial_t U_s^\kappa,\partial_t U_b^\kappa) \|_{H^{s-1}} +\kappa\|(\partial_x  H_s^\kappa,\partial_x H_b^\kappa)\|_{H^{s}} + \kappa^{-1}\|( U_s^\kappa- V_s^\kappa, U_b^\kappa- V_b^\kappa)\|_{H^{s-1}}\\
		 \leq C \|  ( H_s^\kappa, H_b^\kappa, U_s^\kappa, U_b^\kappa, V_s^\kappa, V_b^\kappa) \|_{H^{s}} ,
	\end{multline*}
	where the multiplicative constant $C$ depends on $\|  ( H_s^\kappa, H_b^\kappa, U_s^\kappa, U_b^\kappa, V_s^\kappa, V_b^\kappa) \|_{H^{s_0}}$ and $\inf_{\R} \barH_\ell+H_\ell^\kappa>0$ for ${\ell\in\{s,b\}}$.
	
	We may thus apply Lemma~\ref{L.estimate-linearized}, and infer that we can set  $C$  depending on $s,s_0,M_0,\varsigma$, and $c$ depending only on $\varsigma$, so that as long as
		\begin{equation} \label{eq.H1kappa}
		\forall x\in\R,\quad 	 ( \rho_s,\rho_b,\barH_s+H_s^\kappa(t,x), \barH_b+H_b^\kappa(t,x), \barU_s+U_s^\kappa(t,x), \barU_b+U_b^\kappa(t,x)) \in \mathfrak p^{\varsigma/(2c_0)},
	\end{equation}
	and
	\begin{equation} \label{eq.H2kappa}
		\|  ( H_s^\kappa, H_b^\kappa, U_s^\kappa, U_b^\kappa, V_s^\kappa, V_b^\kappa)(t,\cdot) \|_{H^{s}} 
		\leq 2c_0 M_0
	\end{equation}
	one has, using  that $(\dot H_s,\dot H_b,\dot U_s,\dot U_b,\dot V_s,\dot V_b)\vert_{t=0}=(0,0,0,0,-\kappa \Lambda^{s} (\frac{\partial_x  H_s^\kappa}{\barH_s+H_s^\kappa})\vert_{t=0},-\kappa \Lambda^{s} (\frac{\partial_x  H_b^\kappa}{\barH_b+H_b^\kappa})\vert_{t=0})$, 
	\begin{multline*}
		\|  (\dot H_s,\dot H_b,\dot U_s,\dot U_b,\dot V_s,\dot V_b) (t,\cdot)\|_{L^2}+	c\kappa^{1/2}	\|  (\partial_x\dot H_s,\partial_x\dot H_b,\partial_x\dot V_s,\partial_x\dot V_b) \|_{L^2(0,t;L^2)} \leq \kappa C M_0 \exp(C c_0 M_0 t)
		\\
		+
		C \int_0^t \Big(  c_0M_0\|  (\dot H_s,\dot H_b,\dot U_s,\dot U_b,\dot V_s,\dot V_b,\kappa \partial_x\dot H_s,\kappa \partial_x\dot H_b) (t',\cdot)\|_{L^2} \\
		+ \kappa \|  (\partial_x H_s^\kappa,\partial_x H_b^\kappa,\partial_x^2  H_s,\partial_x^2  H_b,\partial_x^2  V_s,\partial_x^2  V_b,\partial_x U_s,\partial_x U_b,\partial_x V_s,\partial_x V_b) (t',\cdot)\|_{H^s}\Big)
		\exp(C\, c_0M_0\, (t-t')) \dd t.
	\end{multline*}
	Now, we use in the integrand the triangle inequality
	\[  \|  (\partial_x H_s^\kappa,\partial_x H_b^\kappa) (t',\cdot)\|_{H^s} \leq  \|  (\partial_x\dot H_s,\partial_x\dot H_b) (t',\cdot)\|_{L^2}+ \|  (\partial_x H_s,\partial_x H_b) (t',\cdot)\|_{H^s}.\]
	The first contribution may be absorbed by the left-hand side if $\kappa$ is sufficiently small (depending on $c,C,c_0M_0,T_0$), and the second contribution is estimated, as other terms, using the assumption
		\[	\sup\big(\big\{\| ( H_s, H_b, U_s, U_b)(t,\cdot)\|_{H^{s+2}} \  : \ t\in[0,T_0]\big\}\big)\leq c_0 M_0. \]
	Applying then Gronwall's Lemma, we find that 
	\begin{equation}\label{eq.control-difference} \|  (\dot H_s,\dot H_b,\dot U_s,\dot U_b,\dot V_s,\dot V_b) (t,\cdot)\|_{L^2} \leq \kappa C c_0 M_0 K \end{equation}
	where $K$ depends only on $C c_0 M_0 T_0$.
	
	By using again the triangle inequality, we can lower further $\kappa$ (depending on $c_0$) so that~\eqref{eq.control-difference} implies 
	\[
		\forall x\in\R,\quad 	 ( \rho_s,\rho_b,H_s^\kappa(t,x), H_b^\kappa(t,x), U_s^\kappa(t,x), U_b^\kappa(t,x)) \in \mathfrak p^{2\varsigma/(3c_0)}
	\]
	and
	\[
		\|  ( H_s^\kappa, H_b^\kappa, U_s^\kappa, U_b^\kappa, V_s^\kappa, V_b^\kappa)(t,\cdot) \|_{H^{s}} 
		\leq \frac32 c_0 M_0.
	\]
Hence by the usual continuity argument we infer that $(H_s^\kappa,H_b^\kappa,U_s^\kappa,U_b^\kappa)(t,\cdot)$ is well-defined for all $t\in[0,T_0]$, and~\eqref{eq.H1kappa}-\eqref{eq.H2kappa}-\eqref{eq.control-difference} hold. This concludes the proof.
\end{proof}

\section{The hydrostatic Euler equations}\label{S.hydro}

In this section we study the stability of the hydrostatic Euler equations for stratified flows:
\begin{equation} \label{eq.hydro-CV}
\begin{aligned}
\partial_t  h + \partial_x ((1+  h) (  \baru +  u)) & = \kappa \partial_x^2  h,\\
\partial_t  u + \left( \baru+ u - \kappa \frac{\partial_x  h}{1+ h}\right)\partial_x u +\frac1{\barrho} \partial_x  \Psi &=0, \\
\end{aligned}
\end{equation}
where the Montgomery potential $\Psi$ is given by
\begin{equation}\label{eq.def-mont-CV}
\Psi(\cdot,r)= \barrho (r)  \int_{-1}^r  h (\cdot, r') \dd r' + \int_r^0  \barrho (r')  h (\cdot, r') \dd r'  =: (\M[\barrho]   h)(r).
\end{equation}

We recall that our stability results must accommodate with solutions generated by the bilayer system that are piecewise constant and allow the comparison with continuously stratified flows. Hence we must allow for deviations that can be large pointwise, while smallness stems from integration (with respect to the $r$-variable). In practice we shall manipulate simultaneously the pointwise as well as the $L^1_r$ topologies depending on the need. This is the case for instance in the following Lemma, measuring the Lipschitz continuity with respect to the density variable of the Montgomery operator $\M[\barrho]$ defined~\eqref{eq.def-mont-CV}.

\begin{Lem}\label{L.stability-M}
	Let $\barM>0$. There exists $C>0$ such that for any $\barrho_\ell$  ($\ell\in\{1,2\}$) such that 
	\[ \| (\barrho_\ell,\tfrac1{\barrho_\ell})\|_{L^1_r\times L^\infty_r }\leq \barM,\]
	and for any $h\in L^\infty_r$, one has for almost any $r\in (-1,0)$,
	\[  \Big| \Big(\frac1{\barrho_1}\M[\barrho_1]h-\frac1{\barrho_2}\M[\barrho_2]h\Big)(r)\Big|\leq \Big(\barM^3  |\barrho_1-\barrho_2|(r) +\barM \|\barrho_1-\barrho_2\|_{L^1_r} \Big)\|h\|_{L^\infty_r}\]
\end{Lem}
\begin{proof}
	\[\Big(\frac1{\barrho_1}\M[\barrho_1]h-\frac1{\barrho_2}\M[\barrho_2]h\Big)(r) = \big(\frac1{\barrho_1(r)}-\frac1{\barrho_2(r)}\big) \int_r^0  \barrho_1 (r')  h (r') \dd r'+\frac1{\barrho_2(r)} \int_r^0  ( \barrho_1 (r') -\barrho_2 (r')) h (r') \dd r' 
	.\]
\end{proof}

Note also that we are seeking stability estimates for the system~\eqref{eq.hydro-CV}-\eqref{eq.def-mont-CV} with respect to perturbations of the equations ---in particular through $\barrho \approx \barrho_{\rm bl}$ and  $\baru\approx \baru_{\rm bl}$--- and with respect to perturbations of the initial data.
\medskip

In Section~\ref{S.stability-hydro}, we first state a local well-posedness result associated with the initial-value problem for the system~\eqref{eq.hydro-CV}-\eqref{eq.def-mont-CV}, and then provide some stability estimates. These two results by themselves are not sufficient to bootstrap in a standard manner the strong convergence of solutions as the size of deviations shrink, because the topology involved in the first result, namely $L^\infty_r$, is stronger than the topology used in the second result, which is roughly speaking $L^1_r$. For that matter we introduce and study in Section~\ref{S.approx-hydro} a refined approximate solution which, compared with the original reference solution, improves the description of solutions associated with nearby profiles. Specifically, the refined approximate solution satisfies the following three properties:
\begin{enumerate}[i.]
	\item it is well-defined and controlled on a time interval which is uniform with respect to $\kappa\in(0,1]$;
	\item the difference with respect to the nearby solution is controlled for the strong norm associated with $L^\infty_r$;
	\item the difference with respect to the reference solution is controlled for the weak norm associated with $L^1_r$.
\end{enumerate}
The resulting convergence result, Proposition~\ref{P.convergence-hydro}, is stated and proved in Section~\ref{S.convergence-hydro}.

\subsection{Stability estimates}\label{S.stability-hydro}

\begin{Prop}[Well-posedness] \label{P.WP-small-time-hydro}
	Let $s\geq s_0>3/2$, $\varsigma\in(0,1)$, $\barM,M_0>0$ and $c>1$. There exists
	 $T>0$ such that the following holds. 
	
	For all  $\kappa\in(0,1]$,  all $(\barrho,\baru)\in L^\infty((-1,0))$ such that
	\[\| (\baru,\barrho,\tfrac1\barrho)\|_{L^\infty_r }\leq \barM,\]
	and all $( h^0,u^0)\in  L^\infty((-1,0);H^{s}(\R)^2)$ such that for almost all $ r\in (-1,0)$,
	\[  \forall x\in\R,\quad   1+h^0\geq \varsigma ,\]
	 and
	\[ \| ( h^0, u^0)\|_{L^\infty_r H^{s_0}_x}\leq M_0\]
	 there exists a unique $( h,u)\in C([0,T^\star);L^\infty((-1,0);  H^s(\R)^2)$  maximal-in-time strong solution to~\eqref{eq.hydro-CV}--\eqref{eq.def-mont-CV} emerging from the initial data $(h,u)\big\vert_{t=0}=( h^0, u^0)$.
	
	Moreover, $T^\star> \kappa T$ and for any $t\in[0,\kappa T]$ and almost all $ r\in (-1,0)$ one has
	\[  \forall x\in\R,\quad    1+h(t,x,r) \geq \varsigma/c \]
	and
	\[  \max(\{\| ( h,u)\|_{L^\infty(0,t;L^\infty_r H^s_x)},\kappa^{1/2} \| \partial_x  h\|_{L^2(0,t;L^\infty_r H^s_x)}\})\leq c  \| ( h^0,u^0)\|_{L^\infty_r H^s_x}.\]
	
	Moreover, the maximal existence time (resp. the emerging solution in $C([0,T^\star);L^\infty((-1,0);  H^s(\R)^2)$) is a lower semi-continuous (resp. continuous) function of the initial data in $L^\infty((-1,0);H^{s}(\R)^2$ and if $T^\star<\infty$ then
	\[ \| ( h(t,\cdot), u(t,\cdot))\|_{L^\infty_r H^{s_0}_x}\to \infty \text{ as } t\to T^\star.\]
	\begin{proof}
		The proof is very similar to the proof of Proposition~\ref{P.WP-small-time-SV2}, using estimates for transport and transport-diffusion equations pointwisely with respect to the variable $r\in(-1,0)$. The essential arguments are that $L^\infty((-1,0))$ is a Banach algebra and that differentiation with respect to the space variable $\partial_x$ as well as the Fourier multiplier ${\Lambda^s=(\Id-\partial_x^2)^{s/2}}$ commute with the operator $\tfrac1\barrho\M[\barrho]$, and that the linear operator $\tfrac1\barrho\M[\barrho]:L^\infty_r L^2_x\to L^\infty_r L^2_x$ is bounded for any $\barrho\in L^\infty((-1,0))$.
	\end{proof}

\end{Prop}
\begin{Prop}[Stability] \label{P.stability-hydro}
	Let $s>3/2$, $\varsigma\in(0,1)$ and $\barM,M_1,M_2>0$. There exists $C>0$ such that the following holds.

	Let $\kappa\in(0,1]$ and $T>0$ be such that 
	\[  C T\leq \kappa .\]
	Let $\barrho_\ell$, $\baru_\ell$ (for $\ell\in\{1,2\}$) be such that 
	\[ \| (\barrho_\ell,\tfrac1{\barrho_\ell})\|_{L^\infty_r }\leq \barM.\]
	Let	$(h_\ell,u_\ell)$ be solutions to~\eqref{eq.hydro-CV}-\eqref{eq.def-mont-CV} (with $\barrho=\barrho_\ell$ and $\baru=\baru_\ell$) defined on the interval $[0,T]$ and satisfying
	\[ \| (\partial_x h_1,\partial_x u_1)\|_{L^\infty_T L^\infty_r H^s_x}\leq M_1,\qquad  \| (h_2,u_2)\|_{L^\infty_T L^\infty_r H^{s}_x}+ \kappa^{1/2}\, \|\partial_x h_2\|_{L^\infty_r L^2_T  H^s_x}\leq M_2,\]
	and
	\[\essinf_{(t,x,r)\in [0,T]\times \R\times (-1,0)} 1+h_1(t,x,r)\geq \varsigma, \qquad\essinf_{(t,x,r)\in [0,T]\times \R\times (-1,0)} 1+h_2(t,x,r)\geq \varsigma. \]
	Then one has
	\begin{multline}\label{eq.stability-L1} \max(\{\| (h_1-h_2,u_1-u_2)\|_{L^\infty_T L^1_r H^s_x} , \kappa^{1/2} \| \partial_x(h_1-h_2)\|_{L^2_T L^1_r H^s_x}\})\\
		\leq 2 \| (h_1^0-h_2^0,u_1^0-u_2^0)\|_{L^1_r H^s_x} + C T \, \| (\barrho_1-\barrho_2,\baru_1-\baru_2)\|_{L^1_r}
	\end{multline}
	and for almost any $r\in (-1,0)$,
	\begin{multline}\label{eq.stability-pointwise}\max(\{\| (h_1-h_2,u_1-u_2)(\cdot,r)\|_{L^\infty_T H^s_x} ,\kappa^{1/2}\|  \partial_x (h_1-h_2)(\cdot,r)\|_{L^2_T H^s_x}  \}) \\
		\leq 2\big( \| (h_1^0-h_2^0,u_1^0-u_2^0)(r)\|_{ H^s_x}+\| (h_1^0-h_2^0,u_1^0-u_2^0)\|_{L^1_r H^s_x}\big)\\
		 + C T \, \big(|(\barrho_1-\barrho_2,\baru_1-\baru_2)(r)|+\| (\barrho_1-\barrho_2,\baru_1-\baru_2)\|_{L^1_r}\big)  .
	\end{multline}
\end{Prop}
\begin{proof}
	Let us denote $\dot h:=h_1-h_2$, $\dot u:=u_1-u_2$, $\dot\baru:=\baru_1-\baru_2$ and $\dot{\barrho}=\barrho_1-\barrho_2$. We have on the time interval $I:=[0,T]$
	\[
	\begin{cases}
		\partial_t  \dot h + (\baru_2+u_2)\partial_x \dot h- \kappa \partial_x^2 \dot h   =r_1+r_2 ,\\
		\partial_t  \dot u + \left( \baru_2+u_2 +u_2^\star\right)\partial_x\dot u  =r_3, 
	\end{cases}
	\]
	where we denote $u_2^\star:=- \kappa \tfrac{\partial_x  h_2}{1+ h_2}$,
	\[ r_1:= -(\dot \baru+\dot u)\partial_x h_1 - \dot h\partial_x u_1, \quad  r_2:=-(1+h_2)\partial_x \dot u\]
	and
	\[  r_3:= -\left( \dot\baru+\dot u - \kappa \tfrac{\partial_x  \dot h}{1+ h_2}+ \kappa \dot h\tfrac{\partial_x  h_1}{(1+h_1)(1+ h_2)}\right)\partial_x u_1 -\left(\tfrac1{\barrho_1}\M[\barrho_1]-\tfrac1{\barrho_2}\M[\barrho_2]\right)\partial_x  h_1- \tfrac1{\barrho_2}\M[\barrho_2]\partial_x \dot h.\]
	We can now use standard estimates and transport-diffusion and transport equations (\cite{BahouriCheminDanchin11}) to infer that there exists $c_0$ depending only on $s$ such that for almost any $r\in(-1,0)$, one has 
	\begin{multline*} \max(\{\| \dot h(\cdot,r)\|_{L^\infty_T H^s_x},\kappa^{1/2}\|\partial_x \dot h(\cdot,r)\|_{L^2_T H^s_x}\})\leq \Big( \| \dot h(t=0,\cdot,r)\|_{H^s_x} + \| r_1(\cdot,r)\|_{L^1_T H^s_x}+ \kappa^{-1/2}\| r_2(\cdot,r)\|_{L^2_T H^{s-1}_x} \Big)\\
		\times \exp(c_0 \|\partial_x u_2(\cdot,r)\|_{L^1_T H^{s-1}} )
	\end{multline*}
	and
	\[ \| \dot u(\cdot,r)\|_{L^\infty_T H^s_x}\leq \Big( \| \dot u(t=0,\cdot,r)\|_{H^s_x} + \| r_3(\cdot,r)\|_{L^1_T H^s_x} \Big)
		\times \exp(c_0 \|\partial_x u_2(\cdot,r)+\partial_x u_2^\star(\cdot,r)\|_{L^1_T H^{s-1}_x} ).
	\]
	Using that $H^{s'}(\R)$ is a Banach algebra for all $s'>1/2$, we find that for any $t\in[0,T]$
	\[\| r_1(t,\cdot,r)\|_{H^s_x}\lesssim M_1\times \big(|\dot \baru(r)|+\| \dot u(t,\cdot,r)\|_{H^s_x}+\| \dot h(t,\cdot,r)\|_{H^s_x}\big), \quad \| r_2(t,\cdot,r)\|_{H^{s-1}_x}\lesssim (1+M_2)\| \dot u(t,\cdot,r)\|_{H^s_x},\]
	and, making additionally use of standard composition estimates in Sobolev spaces (\cite[Appendix~B]{Lannes}) and  Lemma~\ref{L.stability-M},
	\begin{multline*}\| r_3(t,\cdot,r)\|_{H^s_x}\leq C(s,\barM,M_1,M_2,\varsigma)\Big(|\dot \baru(r)|+\| \dot u(t,\cdot,r)\|_{H^s_x}+ \kappa\| \dot h(t,\cdot,r)\|_{H^{s+1}_x} +|\dot{\barrho}(r)| + \|\dot{\barrho}\|_{L^1_r} \Big)M_1\\
		+\barM^2 \| \partial_x\dot h(t,\cdot)\|_{L^1_r H^s_x}.
	\end{multline*}
	Finally, using product and composition estimates on $u_2^\star=- \kappa\partial_x  h_2+\kappa \tfrac{  h_2}{1+ h_2}\partial_x  h_2$, we have
	\[\|\partial_x u_2^\star(t,\cdot,r)\|_{H^{s-1}_x}\leq C(M_2,\varsigma^{-1}) \, \kappa\, \|\partial_x h_2(t,\cdot ,r)\|_{H^s_x}.\]
	Collecting these estimates we infer that there exists $C>0$ depending only on $s,\barM,M_1,M_2,\varsigma$ such that 
	\begin{multline*}\max(\{\| \dot h(\cdot,r),\dot u(\cdot,r)\|_{L^\infty_T H^s_x},\kappa^{1/2}\|  \partial_x \dot h(\cdot,r)\|_{L^2_T H^s_x}\} )\leq  \exp(C M_2(T+\kappa^{1/2}T^{1/2}))\\
		\times \Big(\| (\dot h(t=0,\cdot,r), \dot u(t=0,\cdot,r))\|_{H^s_x}
		+C\kappa^{-1/2}T^{1/2}\big( \| \dot u(\cdot,r)\|_{L^\infty_T H^s_x}+\kappa^{1/2}\| \partial_x\dot h\|_{L^2_T L^1_r H^s_x}\big) \\
		+CM_1T\times \big(|\dot \baru(r)|+|\dot{\barrho}(r)| + \|\dot{\barrho}\|_{L^1_r} +\| \dot h(\cdot,r)\|_{L^\infty_T H^s_x}+\| \dot u(\cdot,r)\|_{L^\infty_T H^s_x}+\kappa T^{-1/2}\|  \partial_x \dot h(\cdot,r)\|_{L^2_T H^s_x}\big)
		\Big).
	\end{multline*}
	Hence there exists $C>0$, depending only on $s,\barM,M_1,M_2,\varsigma$ such that for any $T$ sufficiently small so that one has $M_1T\leq (36C)^{-1}$, $(\kappa^{-1/2}+\kappa^{1/2}M_1)T^{1/2}\leq (18C)^{-1}$ and  $M_2T+M_2\kappa^{1/2}T^{1/2}\leq C^{-1}\ln(3/2)$ one has
	\begin{multline}\label{eq.stability}
		\frac56\max(\{\| (\dot h(\cdot,r), \dot u(\cdot,r))\|_{L^\infty_T H^s_x},\kappa^{1/2}\|  \partial_x \dot h(\cdot,r)\|_{L^2_T H^s_x} \})\\
		\leq \frac32 \Big(\| (\dot h(t=0,\cdot,r), \dot u(t=0,\cdot,r))\|_{H^s_x}
		+\tfrac1{18}\kappa^{1/2}\| \partial_x\dot h\|_{L^2_T L^1_r H^s_x}\big) 
		+ CM_1T\,\big(|\dot \baru(r)|+|\dot{\barrho}(r)| + \|\dot{\barrho}\|_{L^1_r}\big) 
		\Big).
	\end{multline}
	Integrating this inequality with respect to the variable $r$ and using Minkowski's inequality, we infer the first stability estimate,~\eqref{eq.stability-L1}. Plugging~\eqref{eq.stability-L1} in the right-hand side of~\eqref{eq.stability}, the second stability estimate,~\eqref{eq.stability-pointwise}, follows immediately.
\end{proof}

\begin{Rmk}The restriction $T\lesssim \kappa$ in Propositions~\ref{P.WP-small-time-hydro} and~\ref{P.stability-hydro} is quite stringent. In~\cite{BianchiniDuchene}, some improved stability estimates concerning the system~\eqref{eq.hydro-CV}-\eqref{eq.def-mont-CV} were derived by the authors. The latter estimates exploit a partial symmetric structure of the equations, and demand some extra regularity with respect to the variable $r$. Because we cannot afford such regularity since our stability estimates will be used with piecewise constant functions, we use in the proof of Proposition~\ref{P.stability-hydro} in a stronger way the parabolic regularization of thickness diffusivity. 
	
	In order to obtain a final result on a timescale which is independent of the parameter $\kappa\in(0,1]$ we shall exploit some a priori control on a reference solution and restrict to initial data as well as shear velocity and density distributions that are close to the reference data.
\end{Rmk}

\subsection{Refined approximation}\label{S.approx-hydro}

In this section we consider a given reference solution to the hydrostatic Euler equation~\eqref{eq.hydro-CV}-\eqref{eq.def-mont-CV} and build from it a refined approximate solution associated with nearby profiles. Specifically, let
$\barrho_{\rm ref},\baru_{\rm ref},h_{\rm ref},u_{\rm ref}$ be a solution to
\begin{equation} \label{eq.hydro-ref}
\begin{aligned}
\partial_t  h_{\rm ref} + \partial_x ((1+  h_{\rm ref}) (  \baru_{\rm ref} +  u_{\rm ref})) & = \kappa \partial_x^2  h_{\rm ref},\\
\partial_t  u_{\rm ref} + \left( \baru_{\rm ref}+ u_{\rm ref} - \kappa \frac{\partial_x  h_{\rm ref}}{1+ h_{\rm ref}}\right)\partial_x u_{\rm ref} + \frac1{\barrho_{\rm ref}}\M[\barrho_{\rm ref}]\partial_x  h_{\rm ref} &=0, \\
\end{aligned}
\end{equation}
where we recall that the operator $\M$ is defined in~\eqref{eq.def-mont-CV}. Considering profiles $(\barrho,\baru)$ which are in some sense close to $(\barrho_{\rm ref},\baru_{\rm ref})$ we construct approximate solutions $(h_{\rm app},u_{\rm app})$ as the solutions to
\begin{equation} \label{eq.hydro-app}
\begin{aligned}
\partial_t  h_{\rm app} + \partial_x ((1+  h_{\rm app}) (  \baru +  u_{\rm app})) & = \kappa \partial_x^2  h_{\rm app},\\
\partial_t  u_{\rm app} + \left( \baru+ u_{\rm app} - \kappa \frac{\partial_x  h_{\rm app}}{1+ h_{\rm app}}\right)\partial_x u_{\rm app} &= -\frac1{\barrho}\M[\barrho]\partial_x  h_{\rm ref}.
\end{aligned}
\end{equation}
Notice first that $(h_{\rm app},u_{\rm app})$ satisfies approximately the hydrostatic Euler equations associated with profiles $(\barrho,\baru)$.
\begin{Prop}\label{P.consistency-app}
	For any $s\geq 0$ the refined approximate solution $(h^{\rm app},u^{\rm app})$ satisfies
	\begin{equation} \label{eq.hydro-app-consistency}
	\begin{aligned}
	\partial_t  h_{\rm app} + \partial_x ((1+  h_{\rm app}) (  \baru +  u_{\rm app})) & = \kappa \partial_x^2  h_{\rm app},\\
	\partial_t  u_{\rm app} + \left( \baru+ u_{\rm app} - \kappa \frac{\partial_x  h_{\rm app}}{1+ h_{\rm app}}\right)\partial_x u_{\rm app}+\frac1{\barrho}\M[\barrho]\partial_x  h_{\rm app} &= r_{\rm rem}
	\end{aligned}
	\end{equation}
	with
	\[\| r_{\rm rem}(t,\cdot) \|_{L^\infty_r H^{s}_x} \leq \|\barrho\|_{L^\infty_r}\|\tfrac1{\barrho}\|_{L^\infty_r} \| (h_{\rm ref}-h_{\rm app})(t,\cdot)\|_{L^1_r H^{s+1}_x}
	\]
\end{Prop}
\begin{proof}
	We have
	\begin{align*}
		r_{\rm rem}(\cdot,r)&=\Big(\frac1{\barrho}\M[\barrho](\partial_x  h_{\rm app}-\partial_x  h_{\rm ref})\Big)(\cdot,r)\\
		&=  \int_{-1}^r (\partial_x  h_{\rm app}-\partial_x  h_{\rm ref}) (\cdot, r') \dd r' + \frac1{\barrho(r)}\int_r^0  \barrho (r')(\partial_x  h_{\rm app}-\partial_x  h_{\rm ref}) (\cdot, r') \dd r' ,
	\end{align*}
	and the result follows since $1\leq \sup(\{\barrho (r')/\barrho(r) \ : \ (r,r')\in(-1,0)^2\})\leq \|\barrho\|_{L^\infty_r}\|\tfrac1{\barrho}\|_{L^\infty_r}$.
\end{proof}

That the above remainder term $r_{\rm app}$ is small is a consequence of the subsequent Proposition~\ref{P.convergence-app-bl}. We first prove that for any initial data $(h_{\rm app},u_{\rm app})\vert_{t=0}=(h^0,u^0)$, the emerging solution $(h_{\rm app},u_{\rm app})$ is well-defined and controlled on a time interval uniform with respect to $\kappa\in(0,1]$. 

\begin{Prop}\label{P.control-app} Let $s>3/2$, $\varsigma\in(0,1)$, $\barM,M_0,M_{\rm ref}>0$ and $c>1$. There exists 
	$C>0$ and $T>0$ such that the following holds. 
	
	Let $h_{\rm ref}\in C([0,T_{\rm ref}); L^1((0,1); H^{s+1}_x(\R))) $ be such that
	\[\| \partial_x h_{\rm ref}\|_{L^1_{T_{\rm ref}} L^1_r H^s_x}\leq M_{\rm ref}.\]	
	For all  $\kappa\in(0,1]$,  all $(\barrho,\baru)\in L^\infty((-1,0))$ such that
	\[\| (\barrho,\tfrac1\barrho)\|_{L^\infty_r}\leq \barM, \]
	and all $(h^0,u^0)\in L^\infty((0,1);H^s(\R)^2)$ such that 
	for almost all $ r\in (-1,0)$ 
	\[  \forall x\in\R,\quad   1+h^0\geq \varsigma \]
	 and
	\[\max(\{\|  h^0\|_{L^\infty_r H^{s-1}_x},\kappa^{1/2} \|  h^0 \|_{L^\infty_r H^{s}_x} ,\|  u^0\|_{L^\infty_r H^{s}_x}\})\leq M_0,\]
	there exists a unique 	$(h_{\rm app},u_{\rm app})\in C([0,T^\star); L^\infty((0,1); H^{s}_x(\R)^2))$ maximal solution to~\eqref{eq.hydro-app} emerging from the initial data $(h_{\rm app},u_{\rm app})\vert_{t=0}=(h^0,u^0)$. Moreover one has $T^\star\geq T_{\rm app}:=\min(\{T_{\rm ref},T\})$ and for any $t\in [0,T_{\rm app}]$
	and almost any $r\in(-1,0)$ one has the upper bound
		\begin{multline*}\max(\{\|  h_{\rm app}(\cdot,r)\|_{L^\infty_{T_{\rm app}} H^{s-1}_x},\kappa^{1/2}  \|  h_{\rm app}(\cdot,r)\|_{L^\infty_{T_{\rm app}} H^{s}_x},\kappa\|\partial_x  h_{\rm app}(\cdot,r)\|_{L^2_{T_{\rm app}} H^{s}_x}\}) +\|u_{\rm app}(\cdot,r)\|_{L^\infty_{T_{\rm app}} H^{s}_x} \\
		\leq c 	\max(\{\|  h^0(\cdot,r)\|_{H^{s-1}_x},\kappa^{1/2} \|  h^0(t=0,\cdot,r)\|_{H^{s}_x} , \|  u^0(\cdot,r)\|_{H^{s}_x} \})+CM_{\rm ref}.
	\end{multline*}
\end{Prop}
\begin{proof}
	The existence and uniqueness of $(h_{\rm app},u_{\rm app})\in C([0,T^\star); L^\infty((0,1); H^{s}_x(\R)^2))$ maximal solution to~\eqref{eq.hydro-app} is obtained as in Proposition~\ref{P.WP-small-time-hydro}. We set $T\in [0,T^\star)$. 
	By standard estimates on transport and transport-diffusion equations (\cite{BahouriCheminDanchin11}) applied to~\eqref{eq.hydro-app} we have
		\begin{multline*} \max(\{\|  h_{\rm app}(\cdot,r)\|_{L^\infty_T H^{s-1}_x},\kappa^{1/2}\|\partial_x  h_{\rm app}(\cdot,r)\|_{L^2_T H^{s-1}_x}\})\leq \Big( \|  h^0(\cdot,r)\|_{H^{s-1}_x} + \| r_{\rm app}(\cdot,r)\|_{L^1_T H^{s-1}_x} \Big)\\
		\times \exp(c_0 \|\partial_x u_{\rm app}(\cdot,r)\|_{L^1_T H^{s-1})} ),
	\end{multline*}	
	\begin{multline*} \max(\{\|  h_{\rm app}(\cdot,r)\|_{L^\infty_T H^{s}_x},\kappa^{1/2}\|\partial_x  h_{\rm app}(\cdot,r)\|_{L^2_T H^{s}_x}\})\leq \Big( \|  h^0(\cdot,r)\|_{H^{s}_x}  + \kappa^{-1/2}\| r_{\rm app}(\cdot,r)\|_{L^2_T H^{{s}-1}_x} \Big)\\
		\times \exp(c_0 \|\partial_x u_{\rm app}(\cdot,r)\|_{L^1_T H^{s-1})} ),
	\end{multline*}
	and
\[
		\|  u_{\rm app}(\cdot,r)\|_{L^\infty_T H^{s}_x}\leq \Big( \|  u^0(\cdot,r)\|_{H^{s}_x} + \| r_{\rm ref}(\cdot,r)\|_{L^1_T H^{s}_x} \Big)
		\times \exp(c_0 \|\partial_x u_{\rm app}(\cdot,r)+\partial_x u_{\rm app}^\star(\cdot,r)\|_{L^1_T H^{s-1}_x} ),
\]
	where we denote $u_{\rm app}^\star:=- \kappa \tfrac{\partial_x  h_{\rm app}}{1+ h_{\rm app}}$, $r_{\rm app}:=-(1+h_{\rm app})\partial_x u_{\rm app}$, 
	$r_{\rm ref}=-\frac1{\barrho}\M[\barrho]\partial_x  h_{\rm ref}$, and the constant $c_0$ depends only on $s$.

	Now we notice that for almost any $r\in(-1,0)$,		
	\begin{align*}
		\|\partial_x u_{\rm app}^\star(t,\cdot,r)\|_{H^{s-1}_x}&\leq C(\| h_{\rm app}(t,\cdot ,r)\|_{H^{s-1}_x},\varsigma^{-1}) \, \kappa\, \|\partial_x h_{\rm app}(t,\cdot ,r)\|_{H^s_x},\\
		\| r_{\rm app}(t,\cdot,r)\|_{ H^{s-1}_x} &\leq C(\| h_{\rm app}(t,\cdot ,r)\|_{H^{s-1}_x}) \| u_{\rm app}(t,\cdot ,r)\|_{H^s_x},\\
		\| r_{\rm ref}(t,\cdot,r)\|_{ H^{s}_x} &\leq C(\|\barrho\|_{L^\infty_r}\|\tfrac1{\barrho}\|_{L^\infty_r}) \| \partial_x h_{\rm ref}(t,\cdot)\|_{L^1_r H^s_x}.
	\end{align*}
	From this we infer that there exists $C$, depending only on $s,\| h_{\rm app}(\cdot ,r)\|_{L^\infty_TH^{s-1}_x},\varsigma^{-1},\|\barrho\|_{L^\infty_r}\|\tfrac1{\barrho}\|_{L^\infty_r}$ such that
	\begin{multline*}\max(\{\|  h_{\rm app}(\cdot,r)\|_{L^\infty_T H^{s-1}_x},\kappa^{1/2}  \|  h_{\rm app}(\cdot,r)\|_{L^\infty_T H^{s}_x},\kappa\|\partial_x  h_{\rm app}(\cdot,r)\|_{L^2_T H^{s}_x} ,\|u_{\rm app}(\cdot,r)\|_{L^\infty_T H^{s}_x} \}) \\
		\leq	\Big(\max(\{\|  h^0(\cdot,r)\|_{H^{s-1}_x},\kappa^{1/2} \|  h^0(t=0,\cdot,r)\|_{H^{s}_x},\|  u^0(\cdot,r)\|_{H^{s}_x}\}) \\
		+ C(T+ T^{1/2})\| u_{\rm app}(\cdot ,r)\|_{L^\infty_T H^s_x}+C \| \partial_x h_{\rm ref}\|_{L^1_T L^1_r H^s_x} \Big)\\
		\times \exp(C T \| u_{\rm app}(\cdot ,r)\|_{L^\infty_T H^s_x} +C  T^{1/2} \kappa\|\partial_x  h_{\rm app}(\cdot,r)\|_{L^2_T H^{s}_x} ).
	\end{multline*}
	By the standard continuity argument, we find that
	\begin{multline*}
		\max(\{\|  h_{\rm app}(\cdot,r)\|_{L^\infty_{T_{\rm app}} H^{s-1}_x},\kappa^{1/2}  \|  h_{\rm app}(\cdot,r)\|_{L^\infty_{T_{\rm app}} H^{s}_x},\kappa\|\partial_x  h_{\rm app}(\cdot,r)\|_{L^2_{T_{\rm app}} H^{s}_x} ,\|u_{\rm app}(\cdot,r)\|_{L^\infty_{T_{\rm app}} H^{s}_x} \})\\
		\leq c 	\, \max(\{\|  h^0(\cdot,r)\|_{H^{s-1}_x},\kappa^{1/2} \|  h^0(t=0,\cdot,r)\|_{H^{s}_x}, \|  u^0(\cdot,r)\|_{H^{s}_x}\}) +CM_{\rm ref}
	\end{multline*}for almost all $r\in (0,1)$ and for all $T_{\rm app}\in [0,T]$ such that 
	\[(T_{\rm app}+ T_{\rm app}^{1/2})M_0 \leq  C^{-1} \]
	where $C$ depends only on $s,\varsigma,\barM,cM_0$. This concludes the proof.
\end{proof}
We conclude this section by investigating the difference between the reference solution and the refined approximate solution.
\begin{Prop}\label{P.convergence-app-bl}
	Let $s>3/2$, $\varsigma\in(0,1)$ and $\barM,M_{\rm ref},M_{\rm app}>0$. There exists $C>0$ such that the following holds.

	Let $\kappa\in(0,1]$, $T>0$  and let $\barrho_{\rm ref}$, $\baru_{\rm ref},\barrho,\baru\in L^\infty_r$ be such that 
	\[ \| (\barrho_{\rm ref},\tfrac1{\barrho_{\rm ref}},\barrho,\tfrac1\barrho)\|_{L^\infty_r }\leq \barM.\]
	Let	$(h_{\rm ref},u_{\rm ref})\in C([0,T];L^\infty((-1,0);  H^{s+1}(\R)^2)$ be a solution to~\eqref{eq.hydro-ref} (that is~\eqref{eq.hydro-CV}-\eqref{eq.def-mont-CV}  with $(\barrho,\baru)=(\barrho_{\rm ref},\baru_{\rm ref})$) defined on the time interval $[0,T]$ and satisfying
	\[ \| ( h_{\rm ref}, u_{\rm ref})\|_{L^\infty_T L^\infty_r H^{s+1}_x}\leq M_{\rm ref}.\]
	Let $(h_{\rm app},u_{\rm app})\in C([0,T];L^\infty((-1,0);  H^{s+1}(\R)^2)$ be solution to~\eqref{eq.hydro-app} defined on the time interval $[0,T]$ and satisfying
	\[ \| (h_{\rm app},u_{\rm app})\|_{L^\infty_T L^\infty_r H^{s}_x}+ \kappa^{1/2}\, \|\partial_x h_{\rm app}\|_{L^\infty_r L^2_T  H^s_x}\leq M_{\rm app},\]
	and such that for all $t\in [0,T]$ and almost all $r\in (-1,0)$
	\[\inf_{x\in \R} 1+h_{\rm ref}(t,x,r)\geq \varsigma, \qquad\inf_{x\in\R} 1+h_{\rm app}(t,x,r)\geq \varsigma. \]
	Then one has
	\begin{multline}\label{eq.stability-pointwise-app}\| (h_{\rm ref}-h_{\rm app},u_{\rm ref}-u_{\rm app})(\cdot,r)\|_{L^\infty_T H^s_x}   
		\leq \Big( \| (h_{\rm ref}-h_{\rm app},u_{\rm ref}-u_{\rm app})(t=0,\cdot,r)\|_{ H^s_x}\\ + C \kappa \, \big(|(\barrho_{\rm ref}-\barrho,\baru_{\rm ref}-\baru)(r)|+\| (\barrho_{\rm ref}-\barrho)\|_{L^1_r}\big) \Big) \exp(CT/\kappa).
	\end{multline}	
\end{Prop}
\begin{proof}		
	Let us denote $\dot h:=h_{\rm ref}-h_{\rm app}$, $\dot u:=u_{\rm ref}-u_{\rm app}$, $\dot\baru:=\baru_{\rm ref}-\baru$ and $\dot{\barrho}=\barrho_{\rm ref}-\barrho$. We have on the time interval $I:=[0,T]$ for which both function are well-defined
	\[
	\begin{cases}
		\partial_t  \dot h + (\baru+u_{\rm app})\partial_x \dot h- \kappa \partial_x^2 \dot h   =r_1+r_2 ,\\
		\partial_t  \dot u + \left( \baru+u_{\rm app} +u_{\rm app}^\star\right)\partial_x\dot u  =r_3, 
	\end{cases}
	\]
	where $u_{\rm app}^\star:=- \kappa \tfrac{\partial_x  h_{\rm app}}{1+ h_{\rm app}}$,
	\[ r_1:= -(\dot \baru+\dot u)\partial_x h_{\rm ref} - \dot h\partial_x u_{\rm ref}, \quad  r_2:=-(1+h_{\rm app})\partial_x \dot u\]
	and
	\[  r_3:= -\left( \dot\baru+\dot u - \kappa \tfrac{\partial_x  \dot h}{1+ h_{\rm app}}+ \kappa \dot h\tfrac{\partial_x  h_{\rm ref}}{(1+h_{\rm ref})(1+ h_{\rm app})}\right)\partial_x u_{\rm ref} -\left(\tfrac1{\barrho_{\rm ref}}\M[\barrho_{\rm ref}]-\tfrac1{\barrho}\M[\barrho]\right)\partial_x  h_{\rm ref}.\]
	We can proceed as in the proof of Proposition~\ref{P.stability-hydro} with some straightforward adjustments as for the contributions of $r_3$ since the contribution $- \tfrac1{\barrho}\M[\barrho]\partial_x \dot h$ is nonexistent.
	
	We infer that there exists $C$ depending only on $s,\varsigma,\barM,M_{\rm app},M_{\rm ref}$ such that for all $\kappa\in(0,1]$ and $T_0\in(0,T]$ such that $  C T_0\leq \kappa $,
	 one has for almost any $r\in (-1,0)$,
	\begin{equation*}
	\max(\{	\| (\dot h,\dot u)(\cdot,r)\|_{L^\infty_{T_0} H^s_x} ,\kappa^{1/2}\|  \partial_x \dot h(\cdot,r)\|_{L^2_{T_0} H^s_x}  \}) 
		\leq 2 \| (\dot h,\dot u)(t=0,\cdot,r)\|_{ H^s_x} 
		+ C^2 T_0 \, \big(|\dot\barrho(r)|+|\dot\baru(r)|+\| \dot\barrho\|_{L^1_r}\big)  .
	\end{equation*}
	Iterating this control on $T_n:=\min(\{n\kappa/C,T\})$, we find that
		\[\| (\dot h,\dot u)(\cdot,r)\|_{L^\infty_{T_n} H^s_x} 
	\leq 2^{n+1}  \| (\dot h,\dot u)(t=0,\cdot,r)\|_{ H^s_x}  + C 2^{n} \kappa  \, \big(|\dot\barrho(r)|+|\dot\baru(r)|+\| \dot\barrho\|_{L^1_r}\big) ,
	\]
	which yields the claimed estimate.	
\end{proof}

\subsection{Convergence}\label{S.convergence-hydro}

We now conclude our analysis with the following stability result for solutions to the hydrostatic Euler equations.
\begin{Prop}[Convergence] 		\label{P.convergence-hydro}
	Let $s>3/2$, $\barM,M_{\rm ref},M_0>0$, and $\kappa\in(0,1]$. Then there exists $C,T$ independent of $\kappa$ and $\delta_0>0$ (depending on $\kappa$)  such that the following holds.
	
	Let $(\barrho_{\rm ref},\baru_{\rm ref})\in L^\infty((-1,0))^2$ and 
	$(h_{\rm ref},u_{\rm ref})\in C([0,T];L^\infty((-1,0);  H^{s+2}(\R)^2)$ be a solution to~\eqref{eq.hydro-ref} (that is~\eqref{eq.hydro-CV}-\eqref{eq.def-mont-CV}  with $(\barrho,\baru)=(\barrho_{\rm ref},\baru_{\rm ref})$) defined on the time interval $[0,T]$ 
	such that
	\[ \| (\barrho_{\rm ref},\tfrac1{\barrho_{\rm ref}})\|_{L^\infty_r }\leq \barM , \quad \| ( h_{\rm ref}, u_{\rm ref})\|_{L^\infty_T L^\infty_r H^{s+2}_x}+\| \partial_x  h_{\rm ref}\|_{L^1_T L^1_r H^{s+2}_x} \leq M_{\rm ref}.\]
	Let $(\barrho,\baru)\in L^\infty((-1,0))^2$ and $(h^0,u^0)\in L^\infty((-1,0); H^{s+2}(\R)^2)$ be such that 
	\[ \| (\barrho,\tfrac1{\barrho})\|_{L^\infty_r }\leq \barM, \quad  \|  (h^0,\kappa^{1/2}\partial_x h^0)\|_{L^\infty_r H^{s+1}_x}+\| u^0\|_{L^\infty_r H^{s+2}_x}\leq M_0\]
	and
	\[  \| (\barrho_{\rm ref}-\barrho,\baru_{\rm ref}-\baru)\|_{L^1_r}+\| (h_{\rm ref}\vert_{t=0}-h^0,u_{\rm ref}\vert_{t=0}-u^0)\|_{L^1_r H^{s+1}} \leq \delta_0.\]

	Then $(h,u)\in C([0,T^\star);L^\infty((-1,0);  H^{s+2}(\R)^2)$ the maximal-in-time solution to ~\eqref{eq.hydro-CV}-\eqref{eq.def-mont-CV} emerging from initial data $(h,u)\vert_{t=0}= (h^0, u^0)$ is defined on the time interval $[0,T]$ and we have for almost all $r\in (-1,0)$,
	\begin{multline}\label{eq.est-convergence} \| (h_{\rm ref}-h,u_{\rm ref}-u)(\cdot,r)\|_{L^\infty_T  H^{s}_x}  
		\leq \Big(\| (h_{\rm ref}-h_,u_{\rm ref}-u)(t=0,\cdot,r)\|_{ H^{s}_x}+ \| (h_{\rm ref}-h,u_{\rm ref}-u)(t=0,\cdot)\|_{ L^1_r H^{s+1}_x}\\
		 +|(\barrho_{\rm ref}-\barrho,\baru_{\rm ref}-\baru)(r)|+\|(\barrho_{\rm ref}-\barrho,\baru_{\rm ref}-\baru)\|_{L^1_r}\Big)
		\times C \exp(CT/\kappa) .
	\end{multline}
\end{Prop}

\begin{proof}
	We first recall that Proposition~\ref{P.convergence-app-bl} provides an estimate on the difference between the reference solution $(h_{\rm ref},u_{\rm ref})\in C([0,T];L^\infty((-1,0);  H^{s+2}(\R)^2)$ and the corresponding approximate solution defined as the solution to the system~\eqref{eq.hydro-app} with initial data $(h_{\rm app},u_{\rm app})\vert_{t=0}= (h^0, u^0)$, $(h_{\rm app},u_{\rm app})\in C([0,T]; L^\infty((-1,0); H^{s+2}_x(\R)^2))$, whose existence and control on the time interval $[0,T]$ (lessening $T$ if necessary) is provided by Proposition~\ref{P.control-app}. 	Specifically we have for almost any $r\in (-1,0)$
	\begin{multline*}\| (h_{\rm ref}-h_{\rm app},u_{\rm ref}-u_{\rm app})(\cdot,r)\|_{L^\infty_T H^{s+1}_x}   
		\leq \Big( \| (h_{\rm ref}-h_{\rm app},u_{\rm ref}-u_{\rm app})(t=0,\cdot,r)\|_{ H^{s+1}_x}\\ + C \kappa \, \big(|(\barrho_{\rm ref}-\barrho,\baru_{\rm ref}-\baru)(r)|+\| (\barrho_{\rm ref}-\barrho)\|_{L^1_r}\big) \Big) \exp(CT/\kappa),
	\end{multline*}	
	where  $C$ depends only on $s,\varsigma,\barM,M_{\rm ref}$ and $M_0$.

	We then consider the difference between the exact solution $(h,u)\in C([0,T^\star); L^\infty((-1,0); H^{s+2}_x(\R)^2))$ ---whose existence is provided by Proposition~\ref{P.WP-small-time-hydro}--- and the approximate solution. 
	By means of the consistency result, Proposition~\ref{P.consistency-app}, we can adapt the proof of Proposition~\ref{P.stability-hydro}
	and we find that
	under the assumptions that
	\begin{equation}\label{eq.condition}
		\| (h,u)\|_{L^\infty_T L^\infty_r H^{s}_x}+ \kappa^{1/2}\, \|\partial_x h\|_{L^\infty_r L^2_T  H^s_x}\leq M\quad \text{ and }\quad \essinf_{(t,x,r)\in [0,T]\times \R\times (-1,0)} 1+h(t,x,r)\geq \varsigma
	\end{equation}
	 there exists $C>0$ depending only on $s,\varsigma,\barM,M_{\rm ref},M$ such that for all $\kappa\in(0,1]$ and $t\in(0,\min(\{T,T^\star\}))$ such that $  C t\leq \kappa $ and for almost any $r\in(-1,0)$ one has
	\begin{multline*}
	\max(\{	\| (h-h_{\rm app},u-u_{\rm app})(\cdot,r)\|_{L^\infty_t  H^s_x} ,\kappa^{1/2}\|  \partial_x (h-h_{\rm app})(\cdot,r)\|_{L^2_t  H^s_x}\})  \\
	\leq 2 \| (h-h_{\rm app},u-u_{\rm app})(t=0,\cdot,r)\|_{L^\infty_r H^s_x} 
		+ C t \, \| (h_{\rm ref}-h_{\rm app})(t,\cdot)\|_{L^1_r H^{s+1}_x} .
	\end{multline*}
	Iterating the stability estimate on $T_n:=\min(\{nT_0,T\})$ we deduce that (augmenting $C$ if necessary) 
	\begin{multline*}
		\max(\{\| (h-h_{\rm app},u-u_{\rm app})\|_{L^\infty_t L^\infty_r H^s_x},\kappa^{1/2}\|  \partial_x (h-h_{\rm app})\|_{L^\infty_r L^2_t  H^s_x}\}) \\
		\leq \exp(CT/\kappa)  \| (h-h_{\rm app},u-u_{\rm app})(t=0,\cdot)\|_{L^\infty_r H^s_x}   
		+ C \kappa  \exp(CT/\kappa) \, \| h_{\rm ref}-h_{\rm app}\|_{L^\infty_T L^1_r H^{s+1}_x} .
	\end{multline*}
	Since by construction $(h_{\rm app},u_{\rm app})\vert_{t=0}=(h,u)\vert_{t=0}$ and using the above control on $  h_{\rm ref}- h_{\rm app}$ we infer (augmenting $C$ if necessary)
	\begin{multline*}\max(\{	\| (h-h_{\rm app},u-u_{\rm app})\|_{L^\infty_t L^\infty_r H^s_x} ,\kappa^{1/2}\|  \partial_x (h-h_{\rm app})\|_{L^\infty_r L^2_t  H^s_x}\})\\
	\leq  C \kappa  \exp(CT/\kappa) \,  \| (h_{\rm ref}-h,u_{\rm ref}-u)(t=0,\cdot)\|_{ L^1_r H^{s+1}_x}  +
	C \big(\kappa  \exp(CT/\kappa)\big)^2  \, \|(\barrho_{\rm ref}-\barrho,\baru_{\rm ref}-\baru)\|_{L^1_r}.
\end{multline*}
	Notice that this estimate implies by triangle inequality an upper bound on $\|(h,u)\|_{L^\infty_t L^\infty_r H^{s}_x}+ \kappa^{1/2}\, \|\partial_x h\|_{L^\infty_r L^2_T  H^s_x}$ which (choosing $M$ sufficiently large and $\delta_0$ sufficiently small) enforces strictly the condition~\eqref{eq.condition}, so that by the standard continuity argument the above holds without restriction on $t\in(0,\min(\{T,T^\star\}))$. By the persistence of regularity stated in Proposition~\ref{P.WP-small-time-hydro} we infer that $T^\star>T$ and the above holds for any $t\in(0,T]$. 
	The estimate~\eqref{eq.est-convergence} then immediately follows from the triangle inequality. 
\end{proof}

\section{Conclusion}\label{S.conclusion}

Thanks to Proposition~\ref{P.uniform-WP} (concerning the well-posedness and control of solutions to the bilayer system) on one hand and Proposition~\ref{P.convergence-hydro} (concerning the control of the deviations of nearby solutions to some given reference solutions) on the other hand, one infers immediately the announced rigorous justification of the propagation in time of the columnar motion and sharp stratification assumptions in near-bilayer situations (within the hydrostatic framework). Specifically, we have the following result.

\begin{Thm}\label{T.CV} 
	Let $s\geq s_0>3/2$, $\barM>0$ ,$M_0>0$, $\varsigma\in(0,1)$, and $\kappa\in(0,1]$. Then there exist $C>0,T>0$ independent of $\kappa$ and $\delta_0>0$ (depending on $\kappa$)  such that the following holds.
	
	Let $(\barrho,\baru)\in L^\infty((-1,0))^2$ and $(h^0,u^0)\in L^\infty((-1,0);   H^{s+2}(\R)^2)$ be such that 
	\[ \| (\barrho,\tfrac1{\barrho})\|_{L^\infty_r }\leq \barM, \quad  \|  (h^0,\kappa^{1/2}\partial_x h^0)\|_{L^\infty_r H^{s+1}_x}+\| u^0\|_{L^\infty_r H^{s+2}_x}\leq M_0\]
	and there exists $(\rho_s,\rho_b, \barH_s,\barH_b,\barU_s,\barU_b)\in\R^6$ such that $\barH_s+\barH_b=1$ and $\barU_s+\barU_b=0$ as well as $( H_s^0, H_b^0, U_s^0, U_b^0)\in H^{s+4}(\R)^2\times H^{s+3}(\R)^2$ such that 
	the hyperbolicity condition holds:
	\[ \forall x\in\R,\quad ( \rho_s,\rho_b,\barH_s+ H_s^0(x), \barH_b + H_b^0(x), \barU_s+ U_s^0(x),\barU_b+ U_b^0(x)) \in \mathfrak p^\varsigma,
	\]
	where $\mathfrak p^\varsigma$ is defined in~\eqref{eq.hyperbolicity-condition}, and
	\[ \| ( H_s^0, H_b^0, U_s^0, U_b^0,\kappa\partial_x H_s^0,\kappa\partial_x H_b^0)\|_{H^{s+3}}\leq M_0\]
	and such that denoting $(\barrho_{\rm bl}^0,\baru_{\rm bl}^0,h_{\rm bl}^0,u_{\rm bl}^0)$ through~\eqref{eq.bilayer-to-continuous-intro} we have
	\[  \| (h_{\rm bl}^0-h^0,u_{\rm bl}^0-u^0)\|_{L^1_r H^{s+1}_x}+\| (\barrho_{\rm bl}-\barrho,\baru_{\rm bl}-\baru)\|_{L^1_r} \leq \delta_0,\]
	then
	\begin{enumerate}
		\item there exists $( H_s,  H_b, U_s, U_b)\in C([0,T];H^{s+3}(\R)^4)$ solution to~\eqref{eq.SV2-intro} emerging from the initial data \[(H_s,H_b,U_s,U_b)\big\vert_{t=0}=( H_s^0,  H_b^0,  U_s^0, U_b^0);\]
		\item there exists  $(h,u)\in C([0,T];L^\infty((-1,0);  H^{s+2}(\R)^2)$ solution to ~\eqref{eq.hydro-intro} emerging from initial data \[(h,u)\vert_{t=0}= (h^0, u^0);\] 
		\item denoting $(\barrho_{\rm bl},\baru_{\rm bl},h_{\rm bl},u_{\rm bl})$ through~\eqref{eq.bilayer-to-continuous-intro} we have for almost all $r\in (-1,0)$,
	\begin{multline*}\| (h_{\rm bl}-h,u_{\rm bl}-u)(\cdot,r)\|_{L^\infty_T  H^{s}_x}  
		\leq \Big(\| (h_{\rm bl}^0-h^0,u_{\rm bl}^0-u^0)(\cdot,r)\|_{ H^{s}_x}+ \| (h_{\rm bl}^0-h^0,u_{\rm bl}^0-u^0)\|_{ L^1_r H^{s+1}_x}\\
		+|(\barrho_{\rm bl}-\barrho,\baru_{\rm bl}-\baru)(r)|+\|(\barrho_{\rm bl}-\barrho,\baru_{\rm bl}-\baru)\|_{L^1_r}\Big)
		\times C \exp(Ct/\kappa) .
	\end{multline*}
\end{enumerate}
	
\end{Thm}

\begin{Rmk} Proposition~\ref{P.convergence-hydro} is not limited to the bilayer framework and applies to {\em any} suitably regular reference solution. Hence it can be combined with results analogous to Proposition~\ref{P.uniform-WP} to provide results analogous to Theorem~\ref{T.CV} in the one-layer and multilayer frameworks.
	
The result analogous to Proposition~\ref{P.uniform-WP} in the one-layer framework (that is associated with the standard shallow water equation with thickness diffusivity that was discussed for instance in~\cite{Gent93}) is stated and proved in~\cite[Sect.~2.2]{AdimPhD}. Notice it requires neither the discussion on the hyperbolicity domain of the non-diffusive equations (since the standard non-cavitation assumption guarantees hyperbolicity) nor the discussion on the parabolic regularization of the total velocity (since the natural symmetrizer behaves well with diffusivity contributions).

A result analogous to Proposition~\ref{P.uniform-WP} in the multilayer framework follows from combining the result of~\cite{Duchene13} with the analysis of Section~\ref{S.SV2-diffusive}. In the former, it is proved that assuming sufficiently small shear velocities is a sufficient condition for the (strict) hyperbolicity of the multilayer system in the stably stratified situation. Notice however that this smallness condition is implicit, and not uniform with respect to the number of layers. 
\end{Rmk}

\paragraph{Acknowledgements} The authors thank the Institute for applied mathematics ``Mauro Picone'' (IAC) for its hospitality when this work was initiated, and Laboratoire Ypatia des Sciences Mathématiques LYSM CNRS-INdAM International Research Laboratory as well as Univ. Rennes funding for its financial support. MA and VD thank the program Centre Henri Lebesgue ANR-11-LABX-0020-0, for fostering an attractive mathematical environment. RB acknowledges fundings from the Italian Ministry of University and Research, PRIN 2022HSSYPN (Teseo).

\end{document}